\documentclass[USenglish]{scrartcl} 

\usepackage[T1]{fontenc}	 
\usepackage{lmodern} 

\usepackage{amsthm} 
\usepackage{bm} 
\usepackage{graphicx} 
\usepackage{tikz} 
\usepackage{booktabs} 
\usepackage{multicol} 
\usepackage{lastpage} 
\usepackage[english]{babel} 
\usepackage{microtype} 
\usepackage{tocstyle} 
\usepackage[utf8]{inputenc}
\usepackage{amssymb}
\usepackage{placeins}
\usepackage{amsmath}

\theoremstyle{plain}
\newtheorem{thm}{\protect\theoremname}
\theoremstyle{plain}
\newtheorem{prop}[thm]{\protect\propositionname}
\theoremstyle{definition}
\newtheorem{example}[thm]{\protect\examplename}
\makeatletter

\usepackage{babel}
\providecommand{\examplename}{Example}
\providecommand{\propositionname}{Proposition}
\providecommand{\theoremname}{Theorem}

\begin{document}

\title{\Huge Kernel Embedding based Variational Approach for Low-dimensional Approximation of Dynamical Systems}
\author{Wenchong Tian \footnote{College of Environmental Science and Engineering, Tongji University, Shanghai 200092, P.~R.~China, E-mail: wenchong@tongji.edu.cn}, Hao Wu \footnote{Correspond Author, School of Mathematical Sciences, Tongji University, Shanghai 200092, P.~R.~China, E-mail: hwu@tongji.edu.cn}}

\maketitle

\begin{abstract}
{Transfer operators such as Perron-Frobenius or Koopman operator play a key role in modeling and analysis of complex dynamical systems, which allow linear representations of nonlinear dynamics by transforming the original state variables to feature spaces.
However, it remains challenging to identify the optimal low-dimensional feature mappings from data.
The variational approach for Markov processes (VAMP) provides a comprehensive framework for the evaluation and optimization of feature mappings based on the variational estimation of modeling errors, but it still suffers from a flawed assumption on the transfer operator and therefore sometimes fails to capture the essential structure of system dynamics.
In this paper, we develop a powerful alternative to VAMP, called kernel embedding based variational approach for dynamical systems (KVAD). By using the distance measure of functions in the kernel embedding space, KVAD effectively overcomes
theoretical and practical limitations
of VAMP. In addition, we develop a data-driven KVAD algorithm for seeking the ideal feature mapping within a subspace spanned by given basis functions, and numerical experiments show that the proposed algorithm can significantly improve the modeling accuracy compared to VAMP.}
\end{abstract}

\section{Introduction}

It has been shown that complex nonlinear processes can be accurately
described by linear models in many science and engineering fields,
including wireless communications \cite{konrad2003a,yue2001Composite},
molecular dynamics \cite{wu2017variational,chodera2014markov,wu2016multiensemble},
fluid dynamics \cite{mezic2013analysis,sharma2016correspondence}
and control theory \cite{brunton2016koopman}, where the linear models
can be expressed by a unified formula
\begin{equation}
\mathbb{E}\left[\mathbf{f}(\mathbf{x}_{t+\tau})\right]=\mathbf{K}^{\top}\mathbb{E}[\mathbf{f}(\mathbf{x}_{t})],\label{eq:linear-model}
\end{equation}
and the expectation operator can be removed for deterministic systems.
In such models, the state variable $\mathbf{x}$ is tranformed into
a feature space by the transformation $\mathbf{f}(\mathbf{x})=\left(f_{1}(\mathbf{x}),\ldots,f_{m}(\mathbf{x})\right)^{\top}$,
and the dynamics with lag time $\tau$ is characterized by a linear
time-invariant system in the feature space. Then, all the dynamical
properties of the system can be quantitatively analyzed after estimating
the transition matrix $\mathbf{K}$ from data via linear regression.
A special case of the linear models is Markov state models \cite{schutte1999direct,prinz2011markov}
for conformational dynamics, which is equivalent to the well-known
Ulam's method \cite{goswami2018constrained,junge2009discretization}.
In a Markov state model, the feature mapping $\mathbf{f}$ consists
of indicator functions of subsets of state and $\mathbf{K}=[K_{ij}]$
represents the transition probability from subset $i$ to subset $j$.
Besides Markov state models and the Ulam's method, a large number
of similar modeling methods, e.g., dynamic mode decomposition \cite{schmid2010dynamic,chen2012variants,tu2014on,Kutz2016Dynamic},
time-lagged independent component analysis (TICA) \cite{molgedey1994separation,perezhernandez2013identification,schwantes2013improvements},
extended dynamic mode decomposition (EDMD) \cite{williams2015a,kevrekidis2016kernel,klus2016on},
Markov transition models \cite{wu2015gaussian}, variational approach
of conformation dynamics (VAC) \cite{noe2013a,nuske2014variational,nuske2016variational},
variational Koopman models \cite{wu2017variational} and their variants
based on kernel ebmeddings \cite{kawahara2016dynamic,kevrekidis2016kernel}
and tensors \cite{junge2009discretization,nuske2016variational},
are proposed by using different feature mappings.

From the perspective of operator theory, all the models in the form
of (\ref{eq:linear-model}) can be interpreted as algebraic representations
of transfer operators of systems, including Frobenius-Perron (FP)
operators and Koopman operators, and some of them are universal approximators
for nonlinear dynamical systems under the assumption that the dimension
of feature space is large enough \cite{korda2018convergence} or the
infinite-dimensional kernel mappings are utilized \cite{song2013kernel}.
However, due to the limitation of computational cost and requirements
of dynamical analysis, the low-dimensional approximation of transfer
operators is still a critical and challenging problem in applications
\cite{klus2018data}.

One common way to solve this problem is to identify the dominant dynamical
structures, e.g., metastable states \cite{deuflhard2005robust,roblitz2013fuzzy},
cycles \cite{conrad2016finding} and coherent sets \cite{fackeldey2019metastable},
and achieve the corresponding low-dimensional representations via
spectral clustering. But this strategy assumes that an accurate high-dimensional
model is known a priori, which is often violated especially for large-scale
systems.

Another way for deterministic systems is to seek the feature mapping
$\mathbf{f}$ by minimizing the regression error of (\ref{eq:linear-model})
under the constraint that the state variable can also be accurately
reconstructed from $\mathbf{f}$ \cite{Li_2017,lusch2018deep}. Notice
that the constraint is necessary, otherwise a trivial but uninformative
model with $\mathbf{f}(\mathbf{x})\equiv1$ and $\mathbf{K}=1$ could
be found. Some similar methods are developed for stochastic systems
by considering (\ref{eq:linear-model}) as a conditional generative
model, where the parameters of $\mathbf{f}$ can be trained based
on the likelihood or the other statistical criteria \cite{wu2018deep,mardt2019deep}.
However, these methods are applicable only if $\mathbf{f}$ are non-negative
functions and usually involves the intractable probability density
estimation.

In recent years, the variational approach has led to great progress
for low-dimensional dynamical modeling, which was first proposed for
time-reversible processes \cite{noe2013a,nuske2014variational,mcgibbon2015variational,wu2017variational}
and extended to non-reversible processes in \cite{wu2020variational}.
In contrast with the other methods, this approach provides a general
and unified framework for data-driven model choice, reduction and
optimization of dynamical systems based on the presented variational
scores related to approximation errors of transfer operators. It can
be easily integrated with deep learning to effectively analyze high-dimensional
time series in an end-to-end manner \cite{mardt2018vampnets,chen2019nonlinear}.
The existing variational principle based methods suffer from two drawbacks:
First, it is necessary to assume that the transfer operator is Hilbert-Schimdt
(HS) or compact as an operator between two weighted $\mathcal{L}^{2}$
spaces so that the maximum values of variational scores exist. But
there is no easy way to test the assumption especially when we do
not have strong prior knowledge regarding the system. Specifically,
it can be proved that the assumption does not hold for most deterministic
systems. Second, even for stochastic systems which satisfies the assumption,
the common variational scores are possibly sensitive to small modeling
variations, which could affect the effectiveness of the variational
approach.

In this work, we introduce a kernel embedding based variational approach
for dynamical systems (KVAD) using the theory of kernel embedding
of functions \cite{smola2007hilbert,song2009hilbert,sriperumbudur2011universality,song2013kernel},
where the modeling error is measured by using the distance between
kernel embeddings of transition densities. The kernel based variational
score in KVAD provides a robust and smooth quantification of differences
between transfer operators, and is proved to be bounded for general
dynamical systems, including deterministic and stochastic systems.
Hence, it can effectively overcome the difficulties of existing variational
methods, and expands significantly the range of applications. Like
the previous variational scores, the kernel based score can also be
consistently estimated from trajectory data without solving any intermediate
statistical problem. Therefore, we develop a data-driven KVAD algorithm
by considering $\mathbf{f}$ as a linear superposition of a given
set of basis functions. Furthermore, we establish a relationship between
KVAD, diffusion maps\cite{coifman2006diffusion} and the singular
components of transfer operators. Finally, the effectiveness the proposed
algorithm is demonstrated by numerical experiments.

\section{Problem formulation and preliminaries}

For a Markovian dynamical system in the state space $\mathbb{M}\subset\mathbb{R}^{D}$
, its dynamics can be completely characterized by the transition density
\begin{equation}
p_{\tau}(\mathbf{x},\mathbf{y})\triangleq\mathbb{P}\left(\mathbf{x}_{t+\tau}=\mathbf{y}|\mathbf{x}_{t}=\mathbf{x}\right)
\end{equation}
and the time evolution of the system state distribution can be formulated
as
\begin{equation}
p_{t+\tau}(\mathbf{y})=(\mathcal{P}_{\tau}p_{t})(\mathbf{y})\triangleq\int p_{t}(\mathbf{x})p_{\tau}(\mathbf{x},\mathbf{y})\mathrm{d}\mathbf{x},\label{eq:PF}
\end{equation}
Here $\mathbf{x}_{t}$ denotes the state of the system at time $t$
and $p_{t}$ is the probability density of $\mathbf{x}_{t}$. The
transfer operator $\mathcal{P}_{\tau}$ is called the \emph{Perron-Frobenius
(PF) operator}\footnote{Another commonly used transfer operator for Markovian dynamics is
Koopman operator \cite{williams2015data}, which describes the evolution
of observables instead of probability densities, and is the dual of
the PF operator. In this paper, we focus only on the PF operator for
convenience of analysis.}, which is a linear but usually infinite-dimensional operator. Notice
that the deterministic dynamics in the form of $\mathbf{x}_{t+\tau}=\Theta_{\tau}(\mathbf{x}_{t})$
is a specific case of the Markovian dynamics, where $p_{\tau}(\mathbf{x},\cdot)=\delta_{\Theta_{\tau}(\mathbf{x})}(\cdot)$
is a Dirac function centered at $\Theta_{\tau}(\mathbf{x})$, and
the corresponding PF operator is given by
\[
\int_{\mathbb{A}}(\mathcal{P}_{\tau}p_{t})(\mathbf{y})\mathrm{d}\mathbf{y}=\int_{\mathbf{x}\in\Theta_{\tau}^{-1}(\mathbb{A})}p_{t}(\mathbf{x})\mathrm{d}\mathbf{x}.
\]

By further assuming that the conditional distribution of $\mathbf{x}_{t+\tau}$
for given $\mathbf{x}_{t}=\mathbf{x}$ can always be represented by
a linear combination of $m$ density basis functions $\mathbf{q}=\left(q_{1},\ldots,q_{m}\right)^{\top}:\mathbb{M}\to\mathbb{R}^{m}$,
we obtain a finite-dimensional approximation of the transition density:
\begin{equation}
\hat{p}_{\tau}(\mathbf{x},\mathbf{y})=\mathbf{f}(\mathbf{x})^{\top}\mathbf{q}(\mathbf{y}).\label{eq:linear-density}
\end{equation}
The feature mapping $\mathbf{f}(\mathbf{x})=\left(f_{1}(\mathbf{x}),\ldots,f_{m}(\mathbf{x})\right)^{\top}$
are real-valued observables of the state $\mathbf{x}_{t}=\mathbf{x}$,
and provide a sufficient statistic for predicting the future states.
Based on this approximation, the time evolution equation (\ref{eq:PF})
of the state distribution can then be transformed into a linear evolution
model of the feature functions $\mathbf{f}$ in the form of (\ref{eq:linear-model})
with the transition matrix

\begin{equation}
\mathbf{K}=\int\mathbf{q}(\mathbf{y})\mathbf{f}(\mathbf{y})^{\top}\mathrm{d}\mathbf{y},\label{eq:K}
\end{equation}
and many dynamical properties of the Markov system, including spectral
components, coherent sets and the stationary distribution, can be
efficiently from the linear model.

It is shown in \cite{korda2018convergence} that Eq.~(\ref{eq:linear-density})
provides a universal approximator of Markovian dynamics if the set
of basis function is rich enough. But in this paper, we focus on a
more practically problem: \emph{Given a small $m$, find $\mathbf{f}$
and $\mathbf{q}$ with $\mathrm{dim}(\mathbf{f})=\mathrm{dim}(\mathbf{q})=m$
such that the modeling error of (\ref{eq:linear-density}) is minimized}.

\subsection{Variational principle for Perron-Frobenius operators\label{subsec:Variational-Principle}}

We now briefly introduce the variational principle for evaluating
the approximation quality of linear models (\ref{eq:linear-model}).
The detailed analysis and derivations can be found in \cite{wu2020variational}.

For simplicity of notation, we assume that the available trajectory
data are organized as
\[
\mathbf{X}=(\mathbf{x}_{1},\ldots,\mathbf{x}_{N})^{\top},\quad\mathbf{Y}=(\mathbf{y}_{1},\ldots,\mathbf{y}_{N})^{\top},
\]
where $\{(\mathbf{x}_{n},\mathbf{y}_{n})\}_{n=1}^{N}$ are set of
all transition pairs occurring in the given trajectories, and we denote
the limits of empirical distributions of $\mathbf{X},\mathbf{Y}$
by $\rho_{0}$ and $\rho_{1}$.

Due to the above analysis, the approximation quality of (\ref{eq:linear-density})
can be evaluated by the difference between the PF operator $\mathcal{\hat{P}}_{\tau}$
deduced from $\hat{p}_{\tau}(\mathbf{x},\mathbf{y})$ and the actual
one. In the variational principle proposed by \cite{wu2020variational},
$\mathcal{P}_{\tau}$ is considered as a mapping from $\mathcal{L}_{\rho_{0}^{-1}}^{2}=\{q|\left\Vert q\right\Vert _{\rho_{0}^{-1}}^{2}=\left\langle q,q\right\rangle _{\rho_{0}^{-1}}<\infty\}$
to $\mathcal{L}_{\rho_{1}^{-1}}^{2}=\{q|\left\Vert q\right\Vert _{\rho_{1}^{-1}}^{2}=\left\langle q,q\right\rangle _{\rho_{1}^{-1}}<\infty\}$
with inner products
\[
\left\langle q,q'\right\rangle _{\rho_{0}^{-1}}\triangleq\int q(\mathbf{x})q'(\mathbf{x})\rho_{0}(\mathbf{x})^{-1}\mathrm{d}\mathbf{x},\quad\left\langle q,q'\right\rangle _{\rho_{1}^{-1}}\triangleq\int q(\mathbf{x})q'(\mathbf{x})\rho_{1}(\mathbf{x})^{-1}\mathrm{d}\mathbf{x}.
\]
From this insight, the Hilbert-Schmidt (HS) norm of the modeling error
can be expressed as a weighted mean square error of conditional distributions
\begin{equation}
\left\Vert \mathcal{\hat{P}}_{\tau}-\mathcal{P}_{\tau}\right\Vert _{\mathrm{HS}}^{2}=\int\rho_{0}(\mathbf{x})\left\Vert \hat{p}_{\tau}(\mathbf{x},\cdot)-p_{\tau}(\mathbf{x},\cdot)\right\Vert _{\rho_{1}^{-1}}^{2}\mathrm{d}\mathbf{x},\label{eq:HS}
\end{equation}
and has the decomposition
\[
\left\Vert \mathcal{\hat{P}}_{\tau}-\mathcal{P}_{\tau}\right\Vert _{\mathrm{HS}}^{2}=-\mathcal{R}\left[\mathbf{f},\mathbf{q}\right]+\left\Vert \mathcal{P}_{\tau}\right\Vert _{\mathrm{HS}}^{2}
\]
with
\[
\mathcal{R}\left[\mathbf{f},\mathbf{q}\right]=\mathrm{tr}\left(2\mathbb{E}_{n}\left[\mathbf{f}(\mathbf{x}_{n})\mathbf{g}(\mathbf{y}_{n})^{\top}\right]-\mathbb{E}_{n}\left[\mathbf{f}(\mathbf{x}_{n})\mathbf{f}(\mathbf{x}_{n})^{\top}\right]\mathbb{E}_{n}\left[\mathbf{g}(\mathbf{y}_{n})\mathbf{g}(\mathbf{y}_{n})^{\top}\right]\right)
\]
for $\mathbf{q}(\mathbf{y})=\mathbf{g}(\mathbf{y})\rho_{1}(\mathbf{y})$.
Here $\mathbb{E}_{n}[\cdot]$ denotes the mean value over all transition
pairs $\{(\mathbf{x}_{n},\mathbf{y}_{n})\}_{n=1}^{N}$ as $N\to\infty$.
Because $\left\Vert \mathcal{P}_{\tau}\right\Vert _{\mathrm{HS}}^{2}$
is a constant independent of modeling and $\mathcal{R}$ can be easily
estimated from data via empirical averaging, we can learn parametric
models of $\mathbf{f}(\mathbf{x})$ and $\mathbf{g}(\mathbf{y})$
by maximizing $\mathcal{R}$ as a variational score, which yields
the variational approach for Markov processes (VAMP) \cite{wu2020variational}.

It can be seen that the variational principle is developed under the
assumption that $\mathcal{P}_{\tau}:\mathcal{L}_{\rho_{0}^{-1}}^{2}\to\mathcal{L}_{\rho_{1}^{-1}}^{2}$
is an HS operator.\footnote{This assumption can be relaxed to compactness of $\mathcal{P}_{\tau}$
for some variants of the variational principle, but the relaxed assumption
is not satisfied by deterministic systems either (see Proposition
\ref{prop:deterministic}).} However, in many practical applications, it is difficult to justify
the assumption for unknown transition densities. Particularly, for
deterministic systems, this assumption does not hold and the maximization
of $\mathcal{R}$ could lead to unreasonable models.
\begin{prop}
\label{prop:deterministic}For a deterministic system $\mathbf{x}_{t+\tau}=\Theta_{\tau}(\mathbf{x}_{t})$,
if $\mathcal{L}_{\rho_{1}^{-1}}^{2}$ is an infinite-dimensional Hilbert
space,
\begin{enumerate}
\item $\mathcal{P}_{\tau}$ is not a compact operator from $\mathcal{L}_{\rho_{0}^{-1}}^{2}$
to $\mathcal{L}_{\rho_{1}^{-1}}^{2}$ and hence not an HS operator,
\item $\mathcal{R}\left[\mathbf{f},\mathbf{q}\right]$ can be maximized
by an arbitrary density basis $\mathbf{q}=(g_{1}\cdot\rho_{1},\ldots,g_{m}\cdot\rho_{1})^{\top}$
with $\mathbb{E}_{n}\left[\mathbf{g}(\mathbf{y}_{n})\mathbf{g}(\mathbf{y}_{n})^{\top}\right]=\mathbf{I}$
and $f_{i}(\mathbf{x})=g_{i}(\Theta_{\tau}(\mathbf{x}))$.
\end{enumerate}
\end{prop}

\begin{proof}
See Appendix \ref{sec:Proof-of-Proposition-deterministic}.
\end{proof}

\subsection{Kernel embedding of functions}

Moving away from dynamical systems for a moment, here we introduce
the theory of kernel embedding of functions \cite{song2009hilbert,song2013kernel},
which will be utilized to address the difficulty of VAMP in Section
\ref{sec:Theory}.

A kernel function $\kappa:\mathbb{M}\times\mathbb{M}\to\mathbb{R}$
is a symmetric and positive definite function, which implicitly defines
a kernel mapping $\varphi$ from $\mathbb{M}$ to a reproducing kernel
Hilbert space $\mathbb{H}$, and the inner product of $\mathbb{H}$
satisfies the reproducing property
\[
\left\langle \varphi(\mathbf{x}),\varphi(\mathbf{y})\right\rangle _{\mathbb{H}}=\kappa(\mathbf{x},\mathbf{y}).
\]
By using the kernel mapping, we can embed a function $q:\mathbb{M}\to\mathbb{R}$
in the Hilbert space $\mathbb{H}$ as
\[
\mathcal{E}q=\int\varphi(\mathbf{x})q(\mathbf{x})\mathrm{d}\mathbf{x}\in\mathbb{H}.
\]
Here $\mathcal{E}$ is an injective mapping for $q\in\mathcal{L}^{1}(\mathbb{M})$
if $\kappa$ is a universal kernel \cite{sriperumbudur2011universality},
and we can then measure the similarity between functions $q$ and
$q'$ by the distance between $\mathcal{E}q$ and $\mathcal{E}q'$:
\begin{eqnarray*}
\left\Vert q-q'\right\Vert _{\mathcal{E}}^{2} & = & \left\langle \mathcal{E}\left(q-q'\right),\mathcal{E}\left(q-q'\right)\right\rangle _{\mathbb{H}}\\
 & = &\iint \kappa(\mathbf{x},\mathbf{y})\left(q(\mathbf{x})-q'(\mathbf{x})\right)\left(q(\mathbf{y})-q'(\mathbf{y})\right)\mathrm{d}\mathbf{x}\mathrm{d}\mathbf{y},
\end{eqnarray*}
where $\left\Vert q\right\Vert _{\mathcal{E}}^{2}\triangleq\left\langle q,q\right\rangle _{\mathcal{E}}$
and $\left\langle q,q'\right\rangle _{\mathcal{E}}\triangleq\left\langle \mathcal{E}q,\mathcal{E}q'\right\rangle _{\mathbb{H}}$.
The most commonly used universal kernel for $\mathbb{M}\subset\mathbb{R}^{D}$
is the Gaussian kernel $\kappa(\mathbf{x},\mathbf{y})=\exp\left(-\left\Vert \mathbf{x}-\mathbf{y}\right\Vert ^{2}/\sigma^{2}\right)$,
where $\sigma$ denotes the bandwidth of the kernel.

In the specific case where both $q$ and $q'$ are probability density
functions, $\left\Vert q-q'\right\Vert _{\mathcal{E}}$ is called
the maximum mean discrepancy (MMD) and can be estimated from samples
of $q,q'$ \cite{song2013kernel}.

\section{Theory\label{sec:Theory}}

In this section, we develop a new variational principle for Markovian
dynamics based on the kernel embedding of trainstion densities.

Assuming that $\kappa:\mathbb{M}\times\mathbb{M}\to\mathbb{R}$ is
a universal kernel and bounded by $B$, $\mathcal{E}p_{\tau}(\mathbf{x},\cdot)$
is also bounded in $\mathbb{H}$ with
\begin{eqnarray*}
\left\Vert p_{\tau}(\mathbf{x},\cdot)\right\Vert _{\mathcal{E}}^{2} & =\iint\kappa(\mathbf{y},\mathbf{y}')p_{\tau}(\mathbf{x},\mathbf{y})p_{\tau}(\mathbf{x},\mathbf{y}')\mathrm{d}\mathbf{y}\mathrm{d}\mathbf{y}'\\
 & =\mathbb{E}_{\mathbf{y},\mathbf{y}'\sim p_{\tau}(\mathbf{x},\cdot)}\left[\kappa(\mathbf{y},\mathbf{y}')\right]\le B.
\end{eqnarray*}
Motivated by this conclusion, we propose a new measure for approximation
errors of PF operators
\begin{equation}
\int\rho_{0}(\mathbf{x})\left\Vert \hat{p}_{\tau}(\mathbf{x},\cdot)-p_{\tau}(\mathbf{x},\cdot)\right\Vert _{\mathcal{E}}^{2}\mathrm{d}\mathbf{x}\label{eq:HS-kernel}
\end{equation}
by replacing the $\mathcal{L}_{\rho_{1}^{-1}}^{2}$ norm with $\left\Vert \cdot\right\Vert _{\mathcal{E}}$
in (\ref{eq:HS}), which is finite for both deterministic and stochastic
systems if $\left\Vert \hat{p}_{\tau}(\mathbf{x},\cdot)\right\Vert _{\mathcal{E}}<\infty$.
In contrast with Eq.~(\ref{eq:HS}), the new measure provides a more
general way to quantify modeling errors of dynamics. Furthermore,
from an application point of view, Eq.~(\ref{eq:HS-kernel}) provides
a more smooth and effective representation of modeling errors of the
conditional distributions.
\begin{example}
\label{exa:distance}Let us consider a one-dimensional system with
$\mathbb{M}=[-5,5]$ and $\rho_{1}(\mathbf{x})=0.1\cdot1_{\mathbf{x}\in\mathbb{M}}$.
Suppose that for a given $\mathbf{x}$, $p_{\tau}(\mathbf{x},\mathbf{y})$
and $\hat{p}_{\tau}(\mathbf{x},\mathbf{y})$ are separately uniform
distributions within $[-0.1,0.1]$ and $[c-0.1,c+0.1]$ as shown in
Fig.~\ref{fig:distance}A, where $c$ is the model parameter. In
VAMP, the approximation error between the two conditional distributions
are calculated as $\left\Vert \hat{p}_{\tau}(\mathbf{x},\cdot)-p_{\tau}(\mathbf{x},\cdot)\right\Vert _{\rho_{1}^{-1}}^{2}$,
and it can be observed from Fig.~\ref{fig:distance}B that this quantity
is a constant independent of $c$ except in a small range $c\in[-0.2,0.2]$.
But kernel embedding based error $\left\Vert \hat{p}_{\tau}(\mathbf{x},\cdot)-p_{\tau}(\mathbf{x},\cdot)\right\Vert _{\mathcal{E}}^{2}$
in (\ref{eq:HS-kernel}) is a smooth function of $c$ and provides
a more reasonable metric for the evaluation of $c$.
\end{example}

\begin{figure}
\begin{centering}
\includegraphics[width=0.8\textwidth]{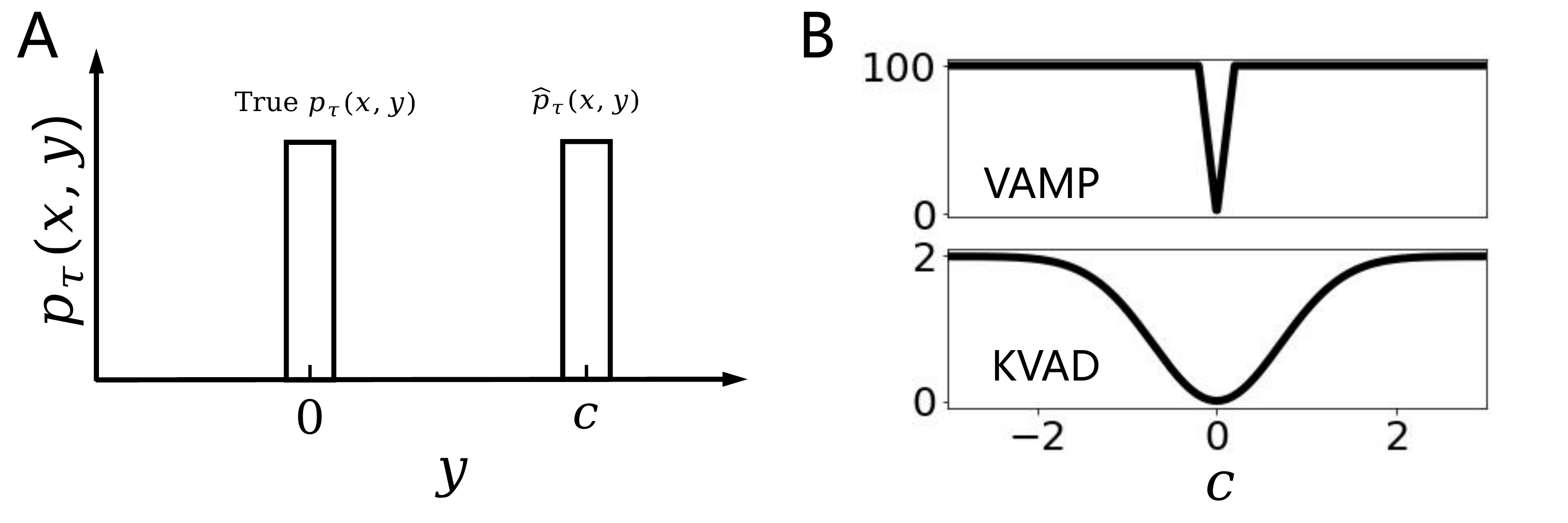}
\par\end{centering}
\caption{\label{fig:distance}Illustration of distribution distances utilized
in VAMP and KVAD. (A) The true conditional density is $p_{\tau}(x,y)=5\cdot1_{\left|y\right|\le0.1}$
for a given $x$, and the approximate density is $\hat{p}_{\tau}(x,y)=5\cdot1_{\left|y-c\right|\le0.1}$
for the same $x$. (B) The distribution distances $\left\Vert \hat{p}_{\tau}(\mathbf{x},\cdot)-p_{\tau}(\mathbf{x},\cdot)\right\Vert _{\rho_{1}^{-1}}^{2}$
defined in VAMP and $\left\Vert \hat{p}_{\tau}(\mathbf{x},\cdot)-p_{\tau}(\mathbf{x},\cdot)\right\Vert _{\mathcal{E}}^{2}$
in KVAD with different values $c$, where $\kappa$ is selected as
the Gaussian kernel with $\sigma=1$.}
\end{figure}

The following proposition shows that Eq.~(\ref{eq:HS-kernel}) can
also be derived from the HS norm of operator error by treating $\mathcal{P}_{\tau}$
as a mapping from $\mathcal{L}_{\rho_{0}^{-1}}^{2}$ to $\mathcal{L}_{\mathcal{E}}^{2}=\left\{ q|\left\Vert q\right\Vert _{\mathcal{E}}^{2}<\infty\right\} $.
\begin{prop}
\label{prop:kvamp}If $\kappa$ is a universal and bounded kernel,
\begin{enumerate}
\item $\mathcal{P}_{\tau}$ is an HS operator from $\mathcal{L}_{\rho_{0}^{-1}}^{2}$
to $\mathcal{L}_{\mathcal{E}}^{2}$, and the corresponding $\left\Vert \mathcal{\hat{P}}_{\tau}-\mathcal{P}_{\tau}\right\Vert _{\mathrm{HS}}^{2}$
is equal to (\ref{eq:HS-kernel}),
\item the HS norm of $\left(\mathcal{\hat{P}}_{\tau}-\mathcal{P}_{\tau}\right)$
satisfies
\[
\left\Vert \mathcal{\hat{P}}_{\tau}-\mathcal{P}_{\tau}\right\Vert _{\mathrm{HS}}^{2}=-\mathcal{R}_{\mathcal{E}}\left[\mathbf{f},\mathbf{q}\right]+\left\Vert \mathcal{P}_{\tau}\right\Vert _{\mathrm{HS}}^{2},
\]
with
\begin{eqnarray*}
\mathcal{R}_{\mathcal{E}}\left[\mathbf{f},\mathbf{q}\right] & =\mathrm{tr}\left(2\mathbf{C}_{fq}-\mathbf{C}_{ff}\mathbf{C}_{qq}\right)
\end{eqnarray*}
for $\hat{p}_{\tau}$ defined by (\ref{eq:linear-density}), where
\begin{eqnarray*}
\mathbf{C}_{ff}   & =& \mathbb{E}_{n}\left[\mathbf{f}(\mathbf{x}_{n})\mathbf{f}(\mathbf{x}_{n})^{\top}\right],\\
\mathbf{C}_{qq} & =& \left[\left\langle q_{i},q_{j}\right\rangle _{\mathcal{E}}\right],\\
\mathbf{C}_{fq}  & =& \left[\left\langle \mathcal{P}_{\tau}\left(f_{i}\rho_{0}\right),q_{j}\right\rangle _{\mathcal{E}}\right],
\end{eqnarray*}
are matrices of size $m\times m$, and $\mathcal{R}_{\mathcal{E}}\left[\mathbf{f},\mathbf{q}\right]$
is called the KVAD score of $\mathbf{f}$ and $\mathbf{q}$.
\end{enumerate}
\end{prop}

\begin{proof}
See Appendix \ref{sec:Proof-of-kvamp}.
\end{proof}
As a result of this proposition, we can find optimal $\mathbf{f}$
and $\mathbf{q}$ by maximizing $\mathcal{R}_{\mathcal{E}}$.

\section{Approximation scheme\label{sec:Approximation-Scheme}}

In this section, we derive a data-driven algorithm to estimate the
optimal low-dimensional linear models based on the variational principle
stated in Proposition \ref{prop:kvamp}.

\subsection{Approximation with fixed $\mathbf{f}$\label{subsec:Approximation-with-fixed-f}}

We first propose a solution for the problem of finding the optimal
$\mathbf{q}$ given that $\mathbf{f}$ is fixed.
\begin{prop}
\label{prop:optimal-q}If $\mathbf{C}_{ff}=\mathbb{E}_{n}\left[\mathbf{f}(\mathbf{x}_{n})\mathbf{f}(\mathbf{x}_{n})^{\top}\right]$
is a full-rank matrix, the solution to $\max_{\mathbf{q}}\mathcal{R}_{\mathcal{E}}\left[\mathbf{f},\mathbf{q}\right]$
is
\begin{equation}
\mathbf{q}(\mathbf{y})=\mathbf{C}_{ff}^{-1}\int\rho_{0}(\mathbf{x})p_{\tau}(\mathbf{x},\mathbf{y})\mathbf{f}(\mathbf{x})\mathrm{d}\mathbf{x}\label{eq:optimal-q}
\end{equation}
and
\begin{eqnarray}
\mathcal{R}_{\mathcal{E}}\left[\mathbf{f}\right] & \triangleq & \max_{\mathbf{q}}\mathcal{R}_{\mathcal{E}}\left[\mathbf{f},\mathbf{q}\right]\nonumber \\
 & = & \mathrm{tr}\left(\mathbf{C}_{ff}^{-1}\mathbb{E}_{n,n'}\left[\mathbf{f}(\mathbf{x}_{n})\kappa(\mathbf{y}_{n},\mathbf{y}_{n'})\mathbf{f}(\mathbf{x}_{n'})^{\top}\right]\right),\label{eq:RE-f}
\end{eqnarray}
where $\mathbb{E}_{n,n'}\left[\cdot\right]$ denotes the expected
value with $(\mathbf{x}_{n},\mathbf{y}_{n})$ and $(\mathbf{x}_{n'},\mathbf{y}_{n'})$
independently drawn from the joint distribution of transition pairs.
\end{prop}

\begin{proof}
See Appendix \ref{sec:Proofs-of-optimal-q}.
\end{proof}
As $\rho_{0}(\mathbf{x})p_{\tau}(\mathbf{x},\mathbf{y})$ in Eq.~(\ref{eq:optimal-q})
is the joint distribution of transition pairs $(\mathbf{x}_{n},\mathbf{y}_{n})$,
we can get a nonparametric approximation of $\mathbf{q}$
\begin{equation}
\mathbf{q}(\mathbf{y})=\frac{1}{N}\sum_{n}\mathbf{C}_{ff}^{-1}\mathbf{f}(\mathbf{x}_{n})\delta_{\mathbf{y}_{n}}(\mathbf{y})\label{eq:optimal-q-data}
\end{equation}
by replacing the the transition pair distribution with its empirical
estimate. This result gives us a linear model (\ref{eq:linear-model})
with transition matrix
\begin{eqnarray}
\mathbf{K} & = & \frac{1}{N}\mathbf{C}_{ff}^{-1}\mathbf{f}(\mathbf{X})^{\top}\mathbf{f}(\mathbf{Y})\nonumber \\
 & = & \mathbf{f}(\mathbf{X})^{+}\mathbf{f}(\mathbf{Y})\label{eq:K-KVAD}
\end{eqnarray}
with $\mathbf{f}(\mathbf{X})=\left(\mathbf{f}(\mathbf{x}_{1}),\ldots,\mathbf{f}(\mathbf{x}_{N})\right)^{\top}\in\mathbb{R}^{N\times m}$
and $\mathbf{f}(\mathbf{X})^{+}$ denoting the Penrose-Moore pseudo-inverse
of $\mathbf{f}(\mathbf{X})$, which is equal to the least square solution
to the regression problem $\mathbf{f}(\mathbf{y}_{n})\approx\mathbf{K}^{\top}\mathbf{f}(\mathbf{x}_{n})$.

\subsection{Approximation with unknown $\mathbf{f}$\label{subsec:Approximation-with-unknown-f}}

We now consider the modeling problem with the normalization condition
\begin{equation}
\int\hat{p}_{\tau}(\mathbf{x},\mathbf{y})\mathrm{d}\mathbf{y}=\int\mathbf{f}(\mathbf{x})^{\top}\mathbf{q}(\mathbf{y})\mathrm{d}\mathbf{y}\equiv1,\label{eq:normalization}
\end{equation}
where $\mathbf{f}$ and $\mathbf{q}$ are both unknown, and we make
the Ansatz to represent $\mathbf{f}$ as linear combinations of basis
functions $\boldsymbol{\chi}=(\chi_{1},\ldots,\chi_{M})^{\top}$.
Furthermore, we assume without loss of generality that the whitening
transformation is applied to the basis functions so that
\begin{equation}
\mathbb{E}_{n}\left[\boldsymbol{\chi}(\mathbf{x}_{n})\right]=\mathbf{0},\quad\mathbb{E}_{n}\left[\boldsymbol{\chi}(\mathbf{x}_{n})\boldsymbol{\chi}(\mathbf{x}_{n})^{\top}\right]=\mathbf{I}.\label{eq:constraint-chi}
\end{equation}
(See, e.g., Appendix F in \cite{wu2020variational} for the details
of whitening transformation.)

It is proved in Appendix \ref{sec:Proof-of-constraint-f} that there
must be a solution to $\max_{\mathbf{f}}\mathcal{R}_{\mathcal{E}}\left[\mathbf{f}\right]$
under constraint (\ref{eq:normalization}) satisfying
\begin{equation}
f_{1}(\mathbf{x})\equiv1,\quad\mathbf{C}_{ff}=\mathbf{I}.\label{eq:constraint-f}
\end{equation}
Therefore, we can model $\mathbf{f}$ in the form of
\begin{equation}
\mathbf{f}(\mathbf{x})=(1,\boldsymbol{\chi}(\mathbf{x})^{\top}\mathbf{U})^{\top},\quad\mathbf{U}\in\mathbb{R}^{m\times M}.\label{eq:ansatz}
\end{equation}
Substituting this Ansatz into the KVAD score, shows that $\mathbf{U}$
can be computed as the solution to the maximization problem:
\begin{eqnarray}
\max_{\mathbf{U}} & \mathcal{R}_{\mathcal{E}}\left(\mathbf{U}\right)\nonumber \\
\mathrm{s.t.} & \mathbf{C}_{ff}=\mathbf{U}^{\top}\mathbf{U}=\mathbf{I}
\end{eqnarray}
with
\begin{equation}
\mathcal{R}_{\mathcal{E}}\left(\mathbf{U}\right)=\frac{1}{N^{2}}\mathrm{tr}\left(\mathbf{U}^{\top}\boldsymbol{\chi}(\mathbf{X})^{\top}\mathbf{G}_{yy}\boldsymbol{\chi}(\mathbf{X})\mathbf{U}\right)+\frac{1}{N^{2}}\mathbf{1}^{\top}\mathbf{G}_{yy}\mathbf{1}\label{eq:RU}
\end{equation}
being a matrix representation of $\mathcal{R}_{\mathcal{E}}\left[\mathbf{f}\right]$.
Here
\[
\mathbf{G}_{yy}=\left[\kappa\left(\mathbf{y}_{i},\mathbf{y}_{j}\right)\right]\in\mathbb{R}^{N\times N}
\]
is the Gram matrix of $\mathbf{Y}$, and $\boldsymbol{\chi}(\mathbf{X})=\left(\boldsymbol{\chi}(\mathbf{x}_{1}),\ldots,\boldsymbol{\chi}(\mathbf{x}_{N})\right)^{\top}\in\mathbb{R}^{N\times M}$.
This problem has the same form as principal component analysis problem
and can be effectively can be solved via the eigendecomposition of
matrix $\boldsymbol{\chi}(\mathbf{X})^{\top}\mathbf{G}_{yy}\boldsymbol{\chi}(\mathbf{X})$
\cite{jolliffe2016principal}. The resulting KVAD algorithm is as
follows, and it can be verified that the normalization conditions
(\ref{eq:normalization}) exactly holds for the estimated transition
density (see Appendix \ref{sec:Normalization-property}).
\begin{enumerate}
\item Select a set of basis function $\boldsymbol{\chi}=(\chi_{1},\ldots,\chi_{M})^{\top}$
with $M\gg m$.
\item Perform the whitening transformation so that (\ref{eq:constraint-chi})
holds.
\item Perform the truncated eigendecomposition
\[
\boldsymbol{\chi}(\mathbf{X})^{\top}\mathbf{G}_{yy}\boldsymbol{\chi}(\mathbf{X})\approx\mathbf{U}\mathbf{S}^{2}\mathbf{U}^{\top},
\]
where $\mathbf{S}=\mathrm{diag}(s_{1},\ldots,s_{m-1})$, $s_{1}\ge s_{2}\ge\ldots\ge s_{m-1}$
are square roots of the largest $m$ eigenvalues of $\boldsymbol{\chi}(\mathbf{X})^{\top}\mathbf{G}_{yy}\boldsymbol{\chi}(\mathbf{X})$,
and $\mathbf{U}=(\mathbf{u}_{1},\ldots,\mathbf{u}_{m-1})$ consists
of the corresponding dominant eigenvectors. This step is a bottleneck
of the algorithm due to the large size Gram matrix $\mathbf{G}_{yy}$,
and the computational cost can be reduced by Nystr\"om approximation
or random fourier features \cite{drineas2005nystrom,rahimi2008random}.
\item Calculate $\mathbf{f}$, $\mathbf{q}$ and $\mathbf{K}$ by (\ref{eq:ansatz},
\ref{eq:optimal-q-data}, \ref{eq:K-KVAD}) with $\mathbf{C}_{ff}=\mathbf{I}$.
\end{enumerate}

\subsection{Component analysis\label{subsec:Component-analysis}}

Due to the fact that $f_{1},q_{1}$ are non-trainable, the approximate
PF operator obtained by the KVAD algorithm can be decomposed as
\begin{eqnarray*}
\hat{\mathcal{P}}_{\tau}q & = & \left\langle f_{1},\rho_{0}\right\rangle _{\rho_{0}^{-1}}q_{1}+\sum_{i=2}^{m}\left\langle q,f_{i}\rho_{0}\right\rangle _{\rho_{0}^{-1}}q_{i}\\
 & = & \left\langle q,\rho_{0}\right\rangle _{\rho_{0}^{-1}}\rho_{1}+\sum_{i=2}^{m}s_{i}\left\langle q,f_{i}\rho_{0}\right\rangle _{\rho_{0}^{-1}}\left(s_{i}^{-1}q_{i}\right).
\end{eqnarray*}
It is worth pointing out that $s_{i},s_{i}^{-1}q_{i+1},f_{i+1}\rho_{0}$
obtained by KVAD algorithm are variational estimates of the $i$th
singular value, left singular function and right singular function
of the operator $\tilde{\mathcal{P}}_{\tau}$ defined by
\begin{equation}
\mathcal{P}_{\tau}q=\left\langle q,\rho_{0}\right\rangle _{\rho_{0}^{-1}}\rho_{1}+\tilde{\mathcal{P}}_{\tau}q,\label{eq:PF-mean-free}
\end{equation}
where $(\tilde{\mathcal{P}}_{\tau}\rho_{0})(\mathbf{y})\equiv0$.
Thus, the KVAD algorithm indeed performs truncated singular value
decomposition (SVD) of $\tilde{\mathcal{P}}_{\tau}$ (see Appendix
\ref{sec:svd}).

At the limit case where the all singular components of $\tilde{\mathcal{P}}_{\tau}$
are exactly estimated by KVAD, we have
\begin{eqnarray}
D_{\tau}\left(\mathbf{x},\mathbf{x}'\right)^{2} & \triangleq & \left\Vert p_{\tau}(\mathbf{x},\cdot)-p_{\tau}(\mathbf{x}',\cdot)\right\Vert _{\mathcal{E}}^{2}\nonumber \\
 & = & \sum_{i=1}^{m-1}s_{i}^{2}\left(f_{i+1}(\mathbf{x})-f_{i+1}(\mathbf{x}')\right)^{2}\label{eq:diffusion-distance}
\end{eqnarray}
for all $\mathbf{x},\mathbf{x}'\in\mathbb{M}$. The distance $D_{\tau}$
measures the dynamical similarity of two points in the state space,
and can be approximated by the Euclidean distance derived from coordinates
$(s_{1}f_{2}(\mathbf{x}),\ldots,s_{m-1}f_{m}(\mathbf{x}))^{\top}$
as shown in (\ref{eq:diffusion-distance}). Hence, KVAD provides an
ideal low-dimensional embedding of system dynamics, and can be reinterpreted
a variant of the diffusion mapping method for dynamical model reduction
\cite{coifman2006diffusion} (see Appendix \ref{sec:Proof-of-distance}).

\section{Relationship with Other Methods}

\subsection{EDMD and VAMP}

It can be seen from (\ref{eq:K-KVAD}) that the optimal linear model
obtained by KVAD is consistent with the model of EDMD \cite{williams2015data}
for given feature functions $\mathbf{f}$. However, the optimization
and the dimension reduction of the observables are not considered
in the conventional EDMD.

Both VAMP \cite{wu2020variational} and KVAD solve this problem by
variational formulations of modeling errors. As analyzed in Sections
\ref{subsec:Variational-Principle} and \ref{sec:Theory}, KVAD is
applicable to more general systems, including deterministic systems,
compared to VAMP. Moreover, VAMP needs to represent both $\mathbf{f}$
and $\mathbf{q}$ by parametric models for dynamical approximation,
whereas KVAD can obtain the optimal $\mathbf{q}$ from data without
any parametric model for given $\mathbf{f}$. Our numerical experiments
(see Section \ref{sec:Experiments}) show that KVAD can often provide
more accurate low-dimensional models than VAMP when the same Ansatz
basis functions are used.

\subsection{Conditional mean embedding, kernel EDMD and kernel CCA}

For given two random variables $\mathbf{x}$ and $\mathbf{y}$, the
conditional mean embedding proposed in \cite{song2009hilbert} characterizes
the conditional distribution of $\mathbf{y}$ for given $\mathbf{x}$
by the conditional embedding operator $\mathcal{C}_{\mathbf{y}|\mathbf{x}}$
with
\[
\mathbb{E}[\varphi(\mathbf{y})|\mathbf{x}]=\mathcal{C}_{\mathbf{y}|\mathbf{x}}\varphi(\mathbf{x}),
\]
where $\varphi$ denotes the kernel mapping and $\mathcal{C}_{\mathbf{y}|\mathbf{x}}$
can be consistently estimated from data. When applied to dynamical
data, this method has the same form as the kernel EDMD and its variants
\cite{schwantes2015modeling,kevrekidis2016kernel,kawahara2016dynamic},
and is indeed a specific case of KVAD with Ansatz functions $\boldsymbol{\chi}=\varphi$
and dimension $m=N$ (see Appendix \ref{sec:Comparison-between-KVAD-KME}).

In addition, for most kernel based dynamical modeling methods, the
dimension reduction problem is not thoroughly investigated. In \cite{klus2020eigendecompositions},
a kernel method for eigendecomposition of transfer operators (including
PF operators and Koopman operators) was developed. But as analyzed
in \cite{wu2020variational}, the dominant eigen-components may not
yield an accurate low-dimensional dynamical model. Kernel canonical
correlation analysis (CCA) \cite{lai2000kernel} can overcome this
problem as a kernelization of VAMP, but it is also applicable only
if $\mathcal{P}_{\tau}$ is a compact operator from $\mathcal{L}_{\rho_{0}^{-1}}^{2}$
to $\mathcal{L}_{\rho_{1}^{-1}}^{2}$.

Compared to the previous kernel methods, KVAD has more flexibility
in model choice, where the dimension and model class of $\mathbf{f}$
can be arbitrarily selected according to practical requirements.

\section{Numerical experiments\label{sec:Experiments}}

In what follows, we demonstrate the benefits of the KVAD method for
studies of nonlinear dynamical systems by two examples, and compare
the results from KVAD with VAMP and kernel EDMD, where the basis functions
in VAMP and kernel functions in kernel EDMD are the same as those
in KVAD. For kernel EDMD, the low-dimensional linear model is achieved
by leading eigenvalues and eigenfunctions as in \cite{klus2020eigendecompositions},
which characterizes invariant subspaces of systems.
\begin{example}
\emph{\label{exa:oscillator}Van der Pol oscillator}, which is a two-dimensional
system governed by
\begin{eqnarray*}
\mathrm{d}x_{t} & = &y_{t}\mathrm{d}t+\xi\cdot\mathrm{d}w_{x,t},\\
\mathrm{d}y_{t} & = &\left(2\left(0.2-x_{t}^{2}\right)y_{t}-x_{t}\right)\mathrm{d}t+\xi\cdot\mathrm{d}w_{y,t},
\end{eqnarray*}
where $w_{x,t}$ and $w_{y,t}$ are standard Wiener processes. The
flow field of this system for $\xi=0$ is depicted in Fig.~\ref{fig:systems}A.
We generate $N=2000$ transition pairs for modeling, where the lag
time $\tau=0.2$, $\mathbf{X}=(\mathbf{x}_{1},\ldots,\mathbf{x}_{N})^{\top}$
are randomly drawn from $[-1.5,1.5]^{2}$, and $\mathbf{Y}=(\mathbf{y}_{1},\ldots,\mathbf{y}_{N})^{\top}$
are obtained by the Euler-Maruyama scheme with step size $0.01$.
\end{example}

\begin{example}
\emph{\label{exa:Lorenz}Lorenz system} defined by
\begin{eqnarray*}
\mathrm{d}x_{t} & = &10(y_{t}-x_{t})\mathrm{d}t+\xi\cdot x_{t}\cdot\mathrm{d}w_{x,t},\\
\mathrm{d}y_{t} & = & \left(28x_{t}-y_{t}-x_{t}z_{t}\right)\mathrm{d}t+\xi\cdot y_{t}\cdot\mathrm{d}w_{y,t},\\
\mathrm{d}z_{t} & = & \left(x_{t}y_{t}-\frac{8}{3}z_{t}\right)\mathrm{d}t+\xi\cdot z_{t}\cdot\mathrm{d}w_{z,t},
\end{eqnarray*}
with $w_{x,t},w_{y,t},w_{z,t}$ being standard Wiener processes. Fig.~\ref{fig:systems}B
plots a trajectory of this system in the state space with $\xi=0$.
We sample $2000$ transition pairs from a simulation trajectory with
length $200$ and lag time $\tau=0.1$ as training data for each $\xi$,
and perform simulations by the Euler-Maruyama scheme with step size
$0.005$.
\end{example}

In both examples, the feature mapping $\mathbf{f}$ is represented
by basis functions
\[
\chi_{i}(\mathbf{x})=\exp\left(-\left(\boldsymbol{\theta}_{i}^{\top}\mathbf{x}+b_{i}\right)^{2}\right), for i=1,\ldots,500
\]
with all components of $\boldsymbol{\theta}_{i}$ and $b_{i}$ randomly
drawn from $[-1,1]$, which are widely used in shallow neural networks
\cite{huang2006extreme,dash2017handbook}. The kernel $\kappa$ is
selected as the Gaussian kernel with $\sigma=1.5$ for the oscillator
and $10$ for the Lorenz system.

Fig.~\ref{fig:systems} shows estimates of singular values of $\tilde{\mathcal{P}}_{\tau}:\mathcal{L}_{\rho_{0}^{-1}}^{2}\to\mathcal{L}_{\mathcal{E}}^{2}$
(KVAD), singular values of $\mathcal{P}_{\tau}:\mathcal{L}_{\rho_{0}^{-1}}^{2}\to\mathcal{L}_{\rho_{1}^{-1}}^{2}$
(VAMP) and absolute values of eigenvalues of $\mathcal{P}_{\tau}$
(kernel EDMD) with different noise parameters, where singular values
must be nonnegative real numbers but eigenvalues could be complex
or negative. We see that the singular values and eigenvalues given
by VAMP and kernel EDMD decay very slowly. Hence, it is difficult
to extract an accurate model from the estimation results of VAMP and
kernel EDMD for a small $m$. Especially for VAMP, a large number
of singular values are close to $1$ as analyzed in Proposition \ref{prop:deterministic}
when the systems are deterministic with $\xi=0$. In contrast, the
singular values utilized in KVAD rapidly converges to zero, which
implies one can effectively extract the essential part of dynamics
from a small number of feature mappings.

The first two singular components of $\tilde{\mathcal{P}}_{\tau}$
for $\xi=0$ obtained by KVAD are shown in Fig.~\ref{fig:singular-components}
(see Section \ref{subsec:Component-analysis}).\footnote{The $q_{i}$ is approximated by multiple delta functions and hard
to visualize.} It can be observed that $f_{1},f_{2}$ of the oscillator characterize
transitions between left-right and up-down areas separately. Those
of the Lorenz systems are related to the two attractor lobes and the
transition areas.

It is interesting to note that the singular values of $\tilde{\mathcal{P}}_{\tau}$
given by KVAD are slightly influenced by $\xi$ as illustrated in
Fig.~\ref{fig:systems}. Our numerical experience also show that
the right singular functions remain almost unchanged for different
$\xi$ (see Fig.~\ref{fig:noise-singular} in Appendix \ref{sec:Estimated-singular-components}).
This phenomenon can be partially explained by the fact that the variational
score $\mathcal{R}_{\epsilon}$ estimated by (\ref{eq:RU}) is not
sensitive to small perturbations of $\mathbf{Y}$ if the bandwidth
of the kernel is large. More thorough investigations on this phenomenon
will be performed in future.

\begin{figure}
\begin{centering}
\includegraphics[width=0.8\textwidth]{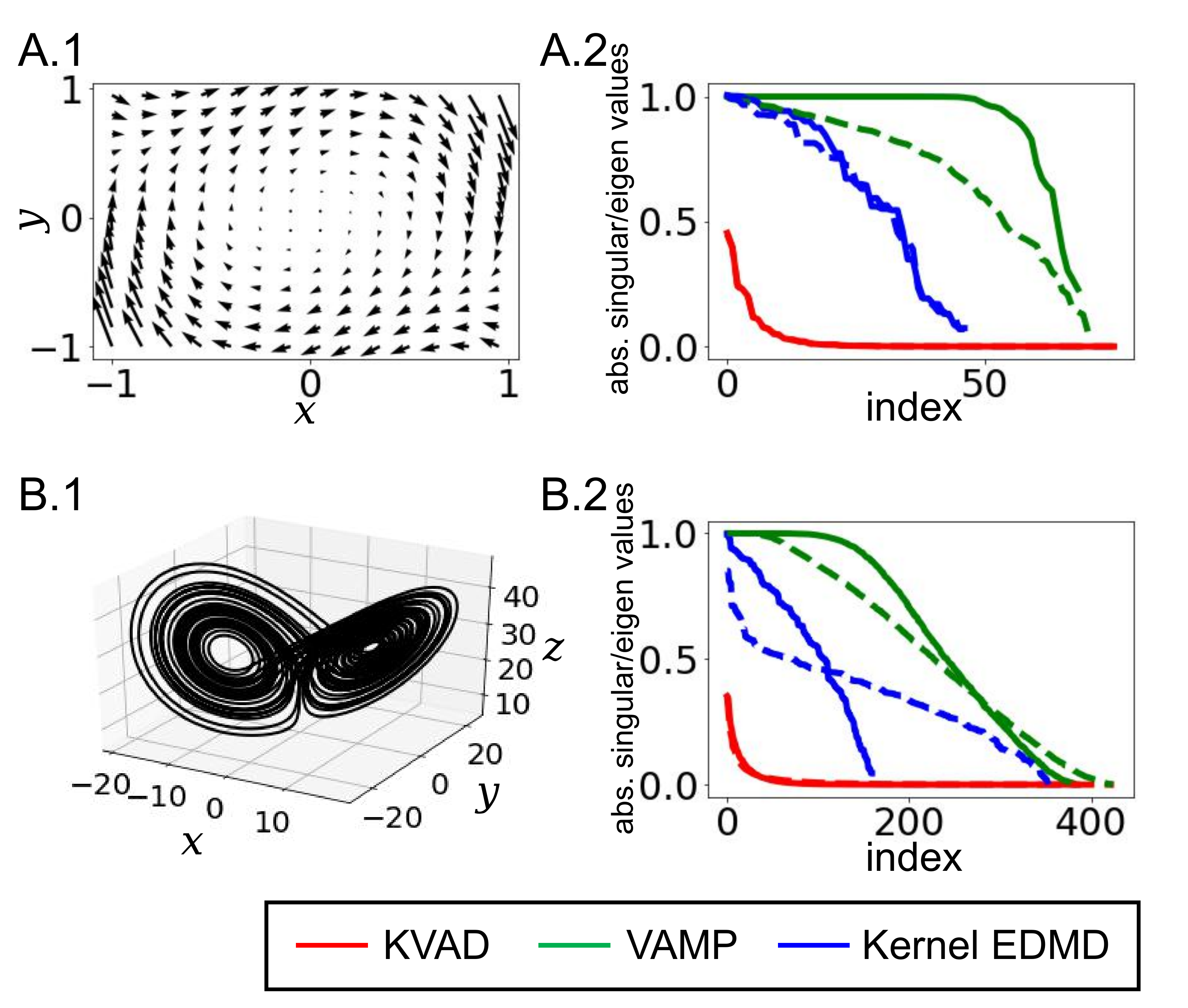}
\par\end{centering}
\caption{\label{fig:systems}(A.1) Flow map of the Van der Pol oscillator,
where the arrows represent directions of $(\mathrm{d}x_{t},\mathrm{d}y_{t})$
with $\xi=0$. (B.1) A typical trajectory of the Lorenz system with
$\xi=0$ generated by the Euler-Maruyama scheme. (A.2 and B.2) Estimated
singular values and absolute values of eigenvalues of the oscillator
and the Lorenz system. Red lines represent singular values of $\tilde{\mathcal{P}}_{\tau}$
estimated by KVAD (see (\ref{eq:PF-mean-free})), green lines singular
values of $\mathcal{P}_{\tau}$ estimated by VAMP, red lines absolute
values of eigenvalues of $\mathcal{P}_{\tau}$ estimated by kernel
EDMD, solid lines estimates with $\xi=0$, and dashed lines those
with $\xi=0.2$ (oscillator) and $0.5$ (Lorenz system). Notice the
total number of spectral components changes in different cases due
to the rank truncation in implementations of SVD and pseudo inverse.}
\end{figure}

\begin{figure}
\begin{centering}
\includegraphics[width=0.8\textwidth]{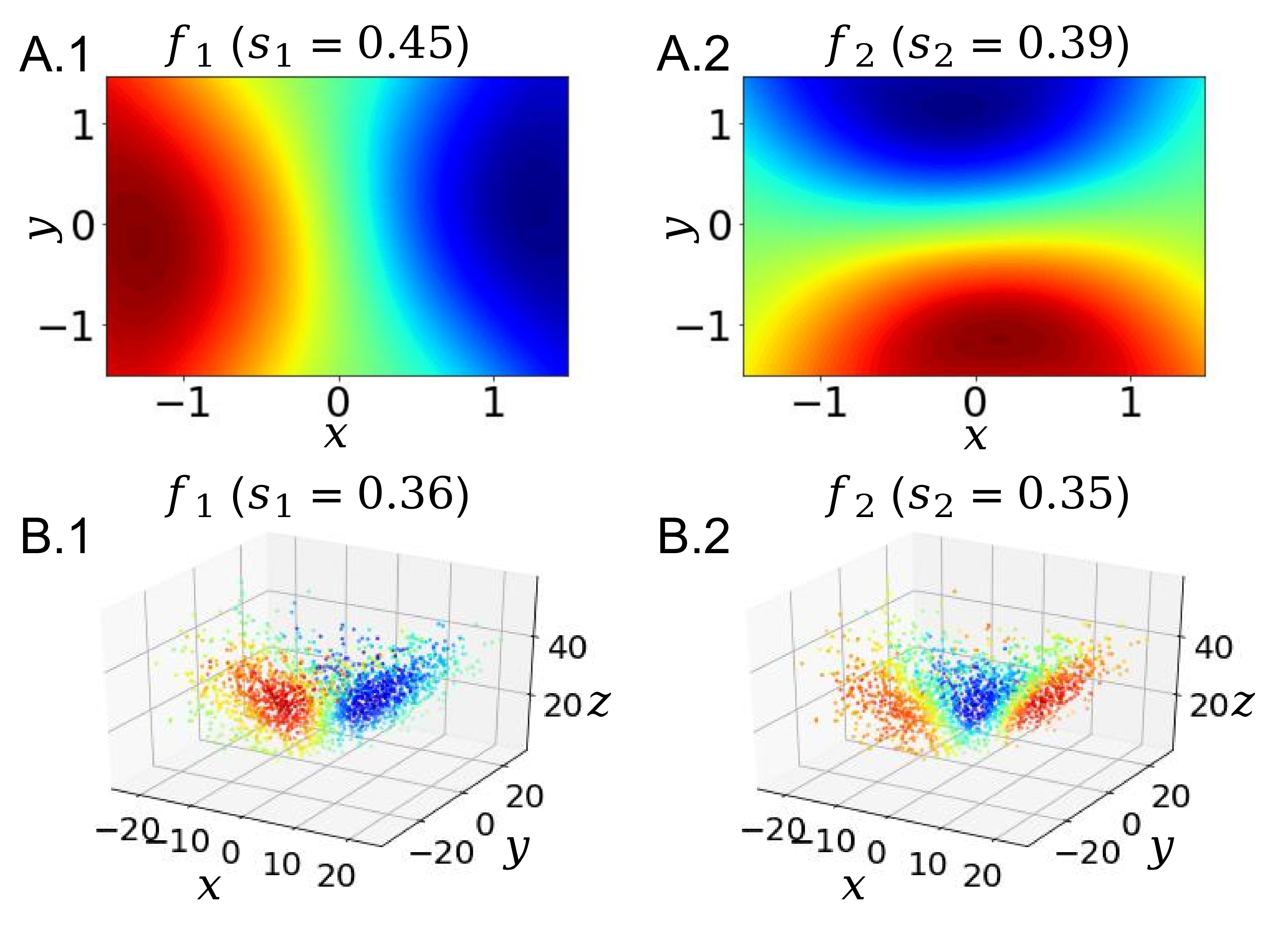}
\par\end{centering}
\caption{\label{fig:singular-components}The first two singular components
computed by KVAD, where $\xi=0$, $s_{i}$ is the estimate of the
$i$th singular value of $\tilde{\mathcal{P}}_{\tau}$ and $f_{i}\cdot\rho_{0}$
the estimate of the corresponding right singular function (see Section
\ref{subsec:Component-analysis}). (A) The oscillator. (B) The Lorenz
system.}
\end{figure}

In order to quantitively evaluate the performance of the three methods,
we define the following trajectory reconstruction error:

\[
\mathrm{error}=\sqrt{\frac{1}{L}\sum_{l=1}^{L}\left\Vert \mathbf{x}_{l\tau}-\mathbb{E}_{\mathrm{model}}[\mathbf{x}_{l\tau}|\mathbf{x}_{0}]\right\Vert },
\]
where $\mathbf{x}_{t}$ is the true trajectory data and $\mathbb{E}_{\mathrm{model}}[\mathbf{x}_{t}|\mathbf{x}_{0}]$
is the conditional mean value of $\mathbf{x}_{t}$ obtained by the
model. The average error over multiple replicate simulations is minimized
if and only if $\mathbb{E}_{\mathrm{model}}[\mathbf{x}_{l\tau}|\mathbf{x}_{0}]$
equals to the exact conditional mean value of $\mathbf{x}_{l\tau}$
for all $l$. For all the three methods,
\begin{eqnarray*}
\mathbb{E}_{\mathrm{model}}[\mathbf{x}_{l\tau}|\mathbf{x}_{0}] & =\mathbf{G}^{\top}\mathbb{E}_{\mathrm{model}}\left[\mathbf{f}(\mathbf{x}_{(l-1)\tau})|\mathbf{x}_{0}\right]\\
 & = &\left(\mathbf{K}^{l-1}\mathbf{G}\right)^{\top}\mathbf{f}(\mathbf{x}_{0}),
\end{eqnarray*}
where $\mathbf{G}$ is the least square solution to the regression
problem $\mathbf{x}_{t+\tau}\approx\mathbf{G}^{\top}\mathbf{f}(\mathbf{x}_{t})$
\cite{williams2015data}. Fig.~\ref{fig:errors} summarizes of reconstruction
errors of the two systems obtained with different choices of the model
dimension $m$ and noise parameter $\xi$, and the superiority of
our KVAD is clearly shown.

\begin{figure}
\begin{centering}
\includegraphics[width=0.8\textwidth]{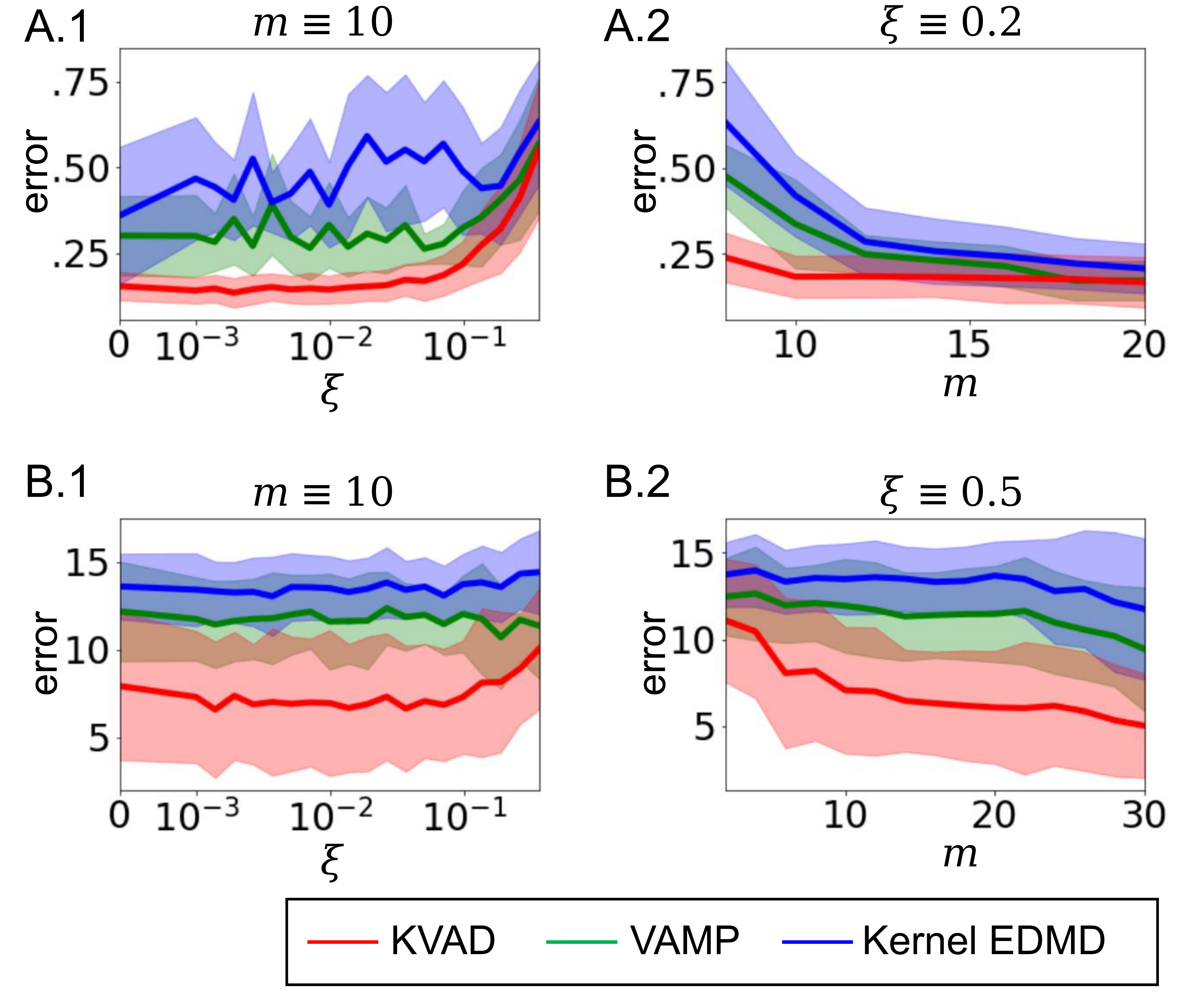}
\par\end{centering}
\caption{\label{fig:errors}(A) Reconstruction errors of the oscillator, where
$L=50$ and $\mathbf{x}_{0}$ is randomly drawn in $[-1.5,1.5]^{2}$.
(B) Reconstruction errors of the Lorenz. Here $L=8$, $\mathbf{x}_{0}$
is randomly sampled from a long simulation trajectory, and the trajectory
is independent of the training data. Error bars represent standard
deviations calculated from 100 bootstrapping replicates of simulations.}
\end{figure}

\section{Conclusion}

In this paper, we combine the kernel embedding technique with the
variational principle for transfer operators. This provides a powerful
and flexible tool for low-dimensional approximation of dynamical system,
and effectively addresses the shortcomings and limitations of the
existing variational approach. In the proposed KVAD framework, a bounded
and well defined distance measure of transfer operators is developed
based on kernel embedding of transition densities, and the corresponding
variational optimization approach can be applied to a broader range
of dynamical systems than the existing variational approaches.

Our future work includes the convergence analysis of KVAD and optimization
of kernel functions. From the algorithmic point of view, the main
remaining question is how to efficiently perform KVAD learning from
big data with deep neural networks. It will be also interesting to
apply KVAD to multi-ensemble Markov models \cite{wu2016multiensemble}
for data analysis of enhanced sampling.

\FloatBarrier

\appendix

\part*{Appendix}

For convenience of notation, we define
\[
\left\langle \mathbf{a},\mathbf{b}^{\top}\right\rangle =\left[\left\langle a_{i},b_{j}\right\rangle \right]\in\mathbb{R}^{n_{1}\times n_{2}}
\]
and
\[
\mathcal{P}_{\tau}\mathbf{a}=(\mathcal{P}_{\tau}a_{1},\ldots,\mathcal{P}_{\tau}a_{n_{1}})^{\top}
\]
for $\mathbf{a}=(a_{1},\ldots,a_{n_{1}})^{\top}$, $\mathbf{b}=(b_{1},\ldots,b_{n_{1}})^{\top}$
and an inner product $\left\langle \cdot,\cdot\right\rangle $.

\section{Proof of Proposition \ref{prop:deterministic}\label{sec:Proof-of-Proposition-deterministic}}

The proof of first conclusion is given in Appendices A.5 and B of
\cite{wu2020variational}, and we prove here the second conclusion.

We first show $\mathcal{R}\left[\mathbf{f},\mathbf{q}\right]\le m$.
Because $\mathbb{E}_{n}\left[\mathbf{f}(\mathbf{x}_{n})\mathbf{f}(\mathbf{x}_{n})^{\top}\right]$
is a positive semi-definite matrices, it can be decomposed as
\begin{eqnarray}
\mathbb{E}_{n}\left[\mathbf{f}(\mathbf{x}_{n})\mathbf{f}(\mathbf{x}_{n})^{\top}\right] &=& \mathbf{Q}\mathbf{D}\mathbf{Q}^{\top},
\end{eqnarray}
where $\mathbf{Q}$ is a orthogonal matrix and $\mathbf{D}$ is a
diagonal matrix. Let $\mathbf{f}'=(f_{1}^{\prime},\ldots,f_{m}^{\prime})^{\top}=\mathbf{Q}^{\top}\mathbf{f}$
and $\mathbf{g}'=(g_{1}^{\prime},\ldots,g_{m}^{\prime})^{\top}=\mathbf{Q}^{\top}\mathbf{g}$,
we have
\begin{eqnarray*}
\mathcal{R}\left[\mathbf{f},\mathbf{q}\right] & = &\mathrm{tr}\left(2\mathbb{E}_{n}\left[\mathbf{Q}^{\top}\mathbf{f}(\mathbf{x}_{n})\mathbf{g}(\mathbf{y}_{n})^{\top}\mathbf{Q}\right] -\mathbb{E}_{n}\left[\mathbf{Q}^{\top}\mathbf{f}(\mathbf{x}_{n})\mathbf{f}(\mathbf{x}_{n})^{\top}\mathbf{Q}\right]\mathbb{E}_{n}\left[\mathbf{Q}^{\top}\mathbf{g}(\mathbf{y}_{n})\mathbf{g}(\mathbf{y}_{n})^{\top}\mathbf{Q}\right]\right)\\
 & = & \mathrm{tr}\left(2\mathbb{E}_{n}\left[\mathbf{f}'(\mathbf{x}_{n})\mathbf{g}'(\mathbf{y}_{n})^{\top}\right]-\mathbb{E}_{n}\left[\mathbf{f}'(\mathbf{x}_{n})\mathbf{f}'(\mathbf{x}_{n})^{\top}\right]\mathbb{E}_{n}\left[\mathbf{g}'(\mathbf{y}_{n})\mathbf{g}'(\mathbf{y}_{n})^{\top}\right]\right)\\
 & = & \sum_{i=1}^{m}2\mathbb{E}_{n}\left[f_{i}^{\prime}(\mathbf{x}_{n})g_{i}^{\prime}(\mathbf{y}_{n})\right]-\mathbb{E}_{n}\left[f_{i}^{\prime}(\mathbf{x}_{n})^{2}\right]\mathbb{E}_{n}\left[g_{i}^{\prime}(\mathbf{x}_{n})^{2}\right]\\
 & \le & \sum_{i=1}^{m}2\mathbb{E}_{n}\left[f_{i}^{\prime}(\mathbf{x}_{n})g_{i}^{\prime}(\mathbf{y}_{n})\right]-\mathbb{E}_{n}\left[f_{i}^{\prime}(\mathbf{x}_{n})g_{i}^{\prime}(\mathbf{y}_{n})\right]^{2}\\
 & \le & m.
\end{eqnarray*}

Under the assumption of $\mathbb{E}_{n}\left[\mathbf{g}(\mathbf{y}_{n})\mathbf{g}(\mathbf{y}_{n})^{\top}\right]=\mathbf{I}$
and $f_{i}(\mathbf{x})=g_{i}(\Theta_{\tau}(\mathbf{x}))$, we have
\begin{eqnarray*}
\left\langle g_{i}\rho_{1},g_{j}\rho_{1}\right\rangle _{\rho_{1}^{-1}} & = &\int\rho_{1}(\mathbf{y})g_{i}(\mathbf{y})g_{j}(\mathbf{y})\mathrm{d}\mathbf{y}\\
 & = & \mathbb{E}_{n}\left[g_{i}(\mathbf{y}_{n})g_{j}(\mathbf{y}_{n})\right]\\
 & = &1_{i=j}
\end{eqnarray*}
and
\begin{eqnarray*}
\mathbb{E}_{n}\left[\mathbf{f}(\mathbf{x}_{n})\mathbf{f}(\mathbf{x}_{n})^{\top}\right] & = & \mathbb{E}_{n}\left[\mathbf{g}(\Theta(\mathbf{x}_{n}))\mathbf{g}(\Theta(\mathbf{x}_{n}))^{\top}\right]\\
 & = & \mathbb{E}_{n}\left[\mathbf{g}(\mathbf{y}_{n})\mathbf{g}(\mathbf{y}_{n})^{\top}\right]=\mathbf{I}
\end{eqnarray*}
with $\mathbf{g}=(g_{1},\ldots,g_{m})^{\top}$ by considering that
$\mathbf{y}_{n}=\Theta_{\tau}(\mathbf{x}_{n})$. Consequently,
\begin{eqnarray*}
\mathcal{R}\left[\mathbf{f},\mathbf{q}\right] & = &\mathrm{tr}\left(2\mathbb{E}_{n}\left[\mathbf{f}(\mathbf{x}_{n})\mathbf{g}(\mathbf{y}_{n})^{\top}\right]-\mathbf{I}\right)\\
 & = &\mathrm{tr}\left(2\mathbb{E}_{n}\left[\mathbf{g}(\mathbf{y}_{n})\mathbf{g}(\mathbf{y}_{n})^{\top}\right]-\mathbf{I}\right)\\
 & = &m,
\end{eqnarray*}
which yields the second conclusion of this proposition.

\section{Proof of Proposition \ref{prop:kvamp}\label{sec:Proof-of-kvamp}}

Let $\{e_{1},e_{2},\ldots\}$ be an orthonormal basis of $\mathcal{L}_{\rho_{0}^{-1}}^{2}$.
We have
\begin{eqnarray*}
\sum_{k}\left\langle \hat{\mathcal{P}}_{\tau}e_{k},\mathcal{P}_{\tau}e_{k}\right\rangle _{\kappa} & = & \sum_{k}\iiiint p_{\tau}(\mathbf{x},\mathbf{y})e_{k}(\mathbf{x})\left\langle \boldsymbol{\varphi}(\mathbf{y}),\boldsymbol{\varphi}(\mathbf{y}')\right\rangle _{\mathbb{H}}\hat{p}_{\tau}(\mathbf{x}',\mathbf{y}')e_{k}(\mathbf{x}')\mathrm{d}\mathbf{x}\mathrm{d}\mathbf{x}'\mathrm{d}\mathbf{y}\mathrm{d}\mathbf{y}'\\
 & = & \iiiint\kappa(\mathbf{y},\mathbf{y}')\cdot\sum_{k}p_{\tau}(\mathbf{x},\mathbf{y})e_{k}(\mathbf{x})\hat{p}_{\tau}(\mathbf{x}',\mathbf{y}')e_{k}(\mathbf{x}')\mathrm{d}\mathbf{x}\mathrm{d}\mathbf{x}'\mathrm{d}\mathbf{y}\mathrm{d}\mathbf{y}'\\
 & = & \iint\kappa(\mathbf{y},\mathbf{y}')\cdot\left(\sum_{k}\iint p_{\tau}(\mathbf{x},\mathbf{y})e_{k}(\mathbf{x})\hat{p}_{\tau}(\mathbf{x}',\mathbf{y}')e_{k}(\mathbf{x}')\mathrm{d}\mathbf{x}\mathrm{d}\mathbf{x}'\right)\mathrm{d}\mathbf{y}\mathrm{d}\mathbf{y}'\\
 & = & \iint\kappa(\mathbf{y},\mathbf{y}')\cdot\sum_{k}\left(\int p_{\tau}(\mathbf{x},\mathbf{y})e_{k}(\mathbf{x})\mathrm{d}\mathbf{x}\right)\\
 &\cdot & \left(\int\hat{p}_{\tau}(\mathbf{x}',\mathbf{y}')e_{k}(\mathbf{x}')\mathrm{d}\mathbf{x}'\right)\mathrm{d}\mathbf{y}\mathrm{d}\mathbf{y}'\\
 & = & \iint\kappa(\mathbf{y},\mathbf{y}')\cdot\left\langle p_{\tau}(\cdot,\mathbf{y})\rho_{0}(\cdot),\hat{p}_{\tau}(\cdot,\mathbf{y})\rho_{0}(\cdot)\right\rangle _{\rho_{0}^{-1}}\mathrm{d}\mathbf{y}\mathrm{d}\mathbf{y}'\\
 & = & \iint\int\rho_{0}(\mathbf{x})\cdot p_{\tau}(\mathbf{x},\mathbf{y})\hat{p}_{\tau}(\mathbf{x},\mathbf{y}')\cdot\kappa(\mathbf{y},\mathbf{y}')\mathrm{d}\mathbf{x}\mathrm{d}\mathbf{y}\mathrm{d}\mathbf{y}'\\
 & = & \int\rho_{0}(\mathbf{x})\left\langle p_{\tau}(\mathbf{x},\cdot),\hat{p}_{\tau}(\mathbf{x},\cdot)\right\rangle _{\mathcal{E}}\mathrm{d}\mathbf{x}.
\end{eqnarray*}
Therefore, $\mathcal{P}_{\tau}$ is an HS operator with
\begin{eqnarray*}
\left\Vert \mathcal{P}_{\tau}\right\Vert _{\mathrm{HS}}^{2} & = & \sum_{k}\left\langle \mathcal{P}_{\tau}e_{k},\mathcal{P}_{\tau}e_{k}\right\rangle _{\mathcal{E}}\\
 & = & \int\rho_{0}(\mathbf{x})\left\Vert p_{\tau}(\mathbf{x},\cdot)\right\Vert _{\mathcal{E}}^{2}\mathrm{d}\mathbf{x}\le B
\end{eqnarray*}
if $\kappa$ is bounded by $B$, and
\begin{eqnarray*}
\left\Vert \mathcal{P}_{\tau}-\hat{\mathcal{P}}_{\tau}\right\Vert _{\mathrm{HS}}^{2} & = & \sum_{k}\left\langle \left(\mathcal{P}_{\tau}-\hat{\mathcal{P}}_{\tau}\right)e_{k},\left(\mathcal{P}_{\tau}-\hat{\mathcal{P}}_{\tau}\right)e_{k}\right\rangle _{\kappa}\\
 & = & \mathcal{D}\left(\mathcal{\hat{P}}_{\tau},\mathcal{P}_{\tau}\right)^{2}.
\end{eqnarray*}

If $\hat{p}_{\tau}(\mathbf{x},\mathbf{y})=\mathbf{f}(\mathbf{x})^{\top}\mathbf{q}(\mathbf{y})$,
we get
\begin{eqnarray*}
\left\Vert \hat{\mathcal{P}}_{\tau}\right\Vert _{\mathrm{HS}}^{2} & = & \iiint\rho_{0}(\mathbf{x})\mathbf{f}(\mathbf{x})^{\top}\mathbf{q}(\mathbf{y})\kappa\left(\mathbf{y},\mathbf{y}'\right)\mathbf{q}(\mathbf{y}')^{\top}\mathbf{f}(\mathbf{x})\mathrm{d}\mathbf{x}\mathrm{d}\mathbf{y}\mathrm{d}\mathbf{y}'\\
 & = & \mathrm{tr}\left(\mathbf{C}_{ff}\mathbf{C}_{qq}\right),
\end{eqnarray*}
\begin{eqnarray*}
\sum_{k}\left\langle \hat{\mathcal{P}}_{\tau}e_{k},\mathcal{P}_{\tau}e_{k}\right\rangle _{\kappa} & = & \int\int\int\rho_{0}(\mathbf{x})\cdot p_{\tau}(\mathbf{x},\mathbf{y})\mathbf{q}(\mathbf{y}')^{\top}\mathbf{f}(\mathbf{x})\cdot\kappa(\mathbf{y},\mathbf{y}')\mathrm{d}\mathbf{x}\mathrm{d}\mathbf{y}\mathrm{d}\mathbf{y}'\\
 & = & \mathrm{tr}\left(\mathbf{C}_{fq}\right),
\end{eqnarray*}
and
\[
\left\Vert \mathcal{P}_{\tau}-\hat{\mathcal{P}}_{\tau}\right\Vert _{\mathrm{HS}}^{2}=-\mathcal{R}_{\mathcal{E}}\left[\mathbf{f},\mathbf{q}\right]+\left\Vert \mathcal{P}_{\tau}\right\Vert _{\mathrm{HS}}^{2}.
\]

\section{Proof of Proposition \ref{prop:optimal-q}\label{sec:Proofs-of-optimal-q}}

Let us first consider the case where $\mathbf{C}_{ff}=\mathbf{I}$.
Then
\begin{eqnarray*}
\mathcal{R_{E}}[\mathbf{f},\mathbf{q}] & = & -\mathrm{tr}\left(\left\langle \mathbf{q},\mathbf{q}^{\top}\right\rangle _{\mathcal{E}}-2\left\langle \mathbf{q},\mathcal{P}_{\tau}\left(\mathbf{f}\rho_{0}\right)^{\top}\right\rangle _{\mathcal{E}}\right)\\
 & = &-\mathrm{tr}\left(\left\langle \mathbf{q}-\mathcal{P}_{\tau}\left(\mathbf{f}\rho_{0}\right),\left(\mathbf{q}-\mathcal{P}_{\tau}\left(\mathbf{f}\rho_{0}\right)\right)^{\top}\right\rangle _{\mathcal{E}}-\left\langle \mathcal{P}_{\tau}\left(\mathbf{f}\rho_{0}\right),\mathcal{P}_{\tau}\left(\mathbf{f}\rho_{0}\right)^{\top}\right\rangle _{\mathcal{E}}\right)\\
 & = & -\sum_{i}\left\Vert q_{i}-\mathcal{P}_{\tau}\left(f_{i}\rho_{0}\right)\right\Vert _{\mathcal{E}}^{2}+\sum_{i}\left\Vert \mathcal{P}_{\tau}\left(f_{i}\rho_{0}\right)\right\Vert _{\mathcal{E}}^{2},
\end{eqnarray*}
which leads to
\begin{eqnarray*}
\arg\max_{\mathbf{q}}\mathcal{R_{E}}[\mathbf{f},\mathbf{q}] & = &\mathcal{P}_{\tau}\left(\mathbf{f}\rho_{0}\right)\\
 & = &\int\rho_{0}(\mathbf{x})p_{\tau}(\mathbf{x},\cdot)\mathbf{f}(\mathbf{x})\mathrm{d}\mathbf{x}
\end{eqnarray*}
and
\begin{eqnarray*}
\mathcal{R_{E}}[\mathbf{f}] & = & \left\Vert \mathcal{P}_{\tau}\left(\mathbf{f}\rho_{0}\right)\right\Vert _{\mathcal{E}}^{2}\\
 & = &\mathrm{tr}\left(\mathbb{E}_{n,n'}\left[\mathbf{f}(\mathbf{x}_{n})\kappa(\mathbf{y}_{n},\mathbf{y}_{n'})\mathbf{f}(\mathbf{x}_{n'})^{\top}\right]\right).
\end{eqnarray*}

We now suppose that $\mathbf{C}_{ff}\neq\mathbf{I}$ and let
\begin{eqnarray*}
\mathbf{f}' & = &\mathbf{C}_{ff}^{-\frac{1}{2}}\mathbf{f},\\
\mathbf{q}' & = &\mathbf{C}_{ff}^{\frac{1}{2}}\mathbf{q}.
\end{eqnarray*}
Becuase $\left\langle \mathbf{f}',\mathbf{f}^{\prime\top}\right\rangle _{\rho_{0}}=\mathbf{I}$,
we have
\[
\arg\max_{\mathbf{q}'}\mathcal{R_{E}}[\mathbf{f}',\mathbf{q}']=\int\rho_{0}(\mathbf{x})p_{\tau}(\mathbf{x},\cdot)\mathbf{f}'(\mathbf{x})\mathrm{d}\mathbf{x}
\]
and
\[
\mathcal{R_{E}}[\mathbf{f}']=\mathrm{tr}\left(\mathbb{E}_{n,n'}\left[\mathbf{f}'(\mathbf{x}_{n})\kappa(\mathbf{y}_{n},\mathbf{y}_{n'})\mathbf{f}'(\mathbf{x}_{n'})^{\top}\right]\right).
\]
Considering that the transition density defined by $(\mathbf{f}',\mathbf{q}')$
is equivalent to that by $(\mathbf{f},\mathbf{q})$ as
\[
\mathbf{f}(\mathbf{x})^{\top}\mathbf{q}(\mathbf{y})=\mathbf{f}'(\mathbf{x})^{\top}\mathbf{q}'(\mathbf{y}),
\]
we can obtain
\begin{eqnarray*}
\arg\max_{\mathbf{q}}\mathcal{R_{E}}[\mathbf{f},\mathbf{q}] & = &\mathbf{C}_{ff}^{-\frac{1}{2}}\mathcal{P}_{\tau}\left(\mathbf{f}'\rho_{0}\right)\\
 & = &\mathbf{C}_{ff}^{-1}\int\rho_{0}(\mathbf{x})p_{\tau}(\mathbf{x},\cdot)\mathbf{f}(\mathbf{x})\mathrm{d}\mathbf{x}
\end{eqnarray*}
and
\begin{eqnarray*}
\mathcal{R_{E}}[\mathbf{f}] & = &\mathcal{R_{E}}[\mathbf{f}']\\
 & = &\mathrm{tr}\left(\mathbb{E}_{n,n'}\left[\mathbf{C}_{ff}^{-\frac{1}{2}}\mathbf{f}(\mathbf{x}_{n})\kappa(\mathbf{y}_{n},\mathbf{y}_{n'})\mathbf{f}(\mathbf{x}_{n'})^{\top}\mathbf{C}_{ff}^{-\frac{1}{2}}\right]\right)\\
 & = &\mathrm{tr}\left(\mathbb{E}_{n,n'}\left[\mathbf{C}_{ff}^{-1}\mathbf{f}(\mathbf{x}_{n})\kappa(\mathbf{y}_{n},\mathbf{y}_{n'})\mathbf{f}(\mathbf{x}_{n'})^{\top}\right]\right).
\end{eqnarray*}

\section{Proof of (\ref{eq:constraint-f})\label{sec:Proof-of-constraint-f}}

Suppose that $(\mathbf{f},\mathbf{q})$ is a solution to $\max\mathcal{R}_{\mathcal{E}}\left[\mathbf{f},\mathbf{q}\right]$
with dimension $m$ under constraint (\ref{eq:normalization}). From
(\ref{eq:normalization}), we have
\[
\mathbf{f}(\mathbf{x})^{\top}\left(\int\mathbf{q}(\mathbf{y})\mathrm{d}\mathbf{y}\right)\equiv1,
\]
which implies that the constant function belongs to the subspace spanned
by $\mathbf{f}$. Thus we can obtain an matrix $\mathbf{R}$ so that
$\mathbf{f}'=(f_{1}^{\prime},\ldots,f_{m}^{\prime})^{\top}=\mathbf{R}\mathbf{f}$
satisfies (\ref{eq:constraint-f}) by Gram-Schmidt orthogonalization,
and $\mathbf{f}'$ and $\mathbf{q}'=\mathbf{R}^{-\top}\mathbf{q}$
also maximizes $\mathcal{R}_{\mathcal{E}}$.

\section{Normalization property of estimated transition density\label{sec:Normalization-property}}

For the transition density $\hat{p}_{\tau}(\mathbf{x},\mathbf{y})=\mathbf{f}(\mathbf{x})^{\top}\mathbf{q}(\mathbf{y})$
obtained by the KVAD algorithm, we have
\[
\int q_{i}(\mathbf{y})\mathrm{d}\mathbf{y}=\frac{1}{N}\sum_{n}f_{i}(\mathbf{x}_{n})=0
\]
for $i>1$. Therefore,
\begin{eqnarray*}
\int\hat{p}_{\tau}(\mathbf{x},\mathbf{y})\mathrm{d}\mathbf{y} & =\mathbf{f}(\mathbf{x})^{\top}\left(\int\mathbf{q}(\mathbf{y})\mathrm{d}\mathbf{y}\right)\\
 & =\mathbf{f}(\mathbf{x})^{\top}(1,0,\ldots,0)^{\top}\\
 & =1.
\end{eqnarray*}

\section{Singular value decomposition of $\tilde{\mathcal{P}}_{\tau}$\label{sec:svd}}

Because $\tilde{\mathcal{P}}_{\tau}$ is also an HS operator from
$\mathcal{L}_{\rho_{0}^{-1}}^{2}$ to $\mathcal{L}_{\mathcal{E}}^{2}$,
there exists the following SVD:
\begin{equation}
\tilde{\mathcal{P}}_{\tau}q=\sum_{i=1}^{\infty}\sigma_{i}\left\langle q,\psi_{i}\right\rangle _{\rho_{0}^{-1}}\phi_{i}.\label{eq:svd}
\end{equation}
Here $\sigma_{i}$ denotes the $i$th largest singular value, and
$\phi_{i},\psi_{i}$ are the corresponding left and right singular
functions. According to the Rayleigh variational principle, for the
$i$th singular component, we have
\begin{equation}
\sigma_{i}^{2}=\max_{q}\left\langle \tilde{\mathcal{P}}_{\tau}q,\tilde{\mathcal{P}}_{\tau}q\right\rangle _{\mathcal{E}}\label{eq:sigmai2}
\end{equation}
under constraints
\begin{equation}
\left\langle q,q\right\rangle _{\rho_{0}^{-1}}=1,\qquad\left\langle q,\psi_{j}\right\rangle _{\rho_{0}^{-1}}=0,\quad\forall j=1,\ldots,i-1\label{eq:orth}
\end{equation}
and the solution is $q=\psi_{i}$.

From the above variational formulation of SVD, we can obtain the following
proposition:
\begin{prop}
The singular functions $\psi_{i}$ of $\tilde{\mathcal{P}}_{\tau}$
satisfies
\[
\left\langle \rho_{0},\psi_{i}\right\rangle _{\rho_{0}^{-1}}=0
\]
if $\sigma_{i}>0$.
\end{prop}

\begin{proof}
We first show that $\left\langle \rho_{0},\psi_{1}\right\rangle _{\rho_{0}^{-1}}=0$
by contradiction. If $\left\langle \rho_{0},\psi_{1}\right\rangle _{\rho_{0}^{-1}}=c_{1}\neq0$,
$\psi_{1}$ can be decomposed as
\[
\psi_{1}=c_{1}\rho_{0}+\tilde{\psi}_{1}.
\]
Because
\begin{eqnarray*}
\left\langle \tilde{\mathcal{P}}_{\tau}\tilde{\psi}_{1},\tilde{\mathcal{P}}_{\tau}\tilde{\psi}_{1}\right\rangle _{\mathcal{E}} & = &\left\langle \tilde{\mathcal{P}}_{\tau}\psi_{1},\tilde{\mathcal{P}}_{\tau}\psi_{1}\right\rangle _{\mathcal{E}},\\
\left\langle \tilde{\psi}_{1},\tilde{\psi}_{1}\right\rangle _{\rho_{0}^{-1}} & = &1-c_{1}^{2},
\end{eqnarray*}
we can get that $\left\langle \tilde{\psi}_{1}^{\prime},\tilde{\psi}_{1}^{\prime}\right\rangle _{\rho_{0}^{-1}}=1$
and
\[
\left\langle \tilde{\mathcal{P}}_{\tau}\tilde{\psi}_{1}^{\prime},\tilde{\mathcal{P}}_{\tau}\tilde{\psi}_{1}^{\prime}\right\rangle _{\mathcal{E}}=\frac{1}{1-c_{1}^{2}}\left\langle \tilde{\mathcal{P}}_{\tau}\psi_{1},\tilde{\mathcal{P}}_{\tau}\psi_{1}\right\rangle _{\mathcal{E}}>\sigma^{2}
\]
with
\[
\tilde{\psi}_{1}^{\prime}=\left(1-c_{1}^{2}\right)^{-\frac{1}{2}}\tilde{\psi}_{1},
\]
which leads to a contradiction. Therefore, $\left\langle \rho_{0},\psi_{1}\right\rangle _{\rho_{0}^{-1}}=0$.

For $\psi_{2}$, we can also show that
\[
\tilde{\psi}_{2}^{\prime}=\left(1-c_{2}^{2}\right)^{-\frac{1}{2}}\left(\psi_{2}-c_{2}\rho_{0}\right)
\]
with $c_{2}=\left\langle \rho_{0},\psi_{2}\right\rangle _{\rho_{0}^{-1}}$
and $\left\langle \psi_{1},\tilde{\psi}_{2}^{\prime}\right\rangle _{\rho_{0}^{-1}}=0$
is a better solution than $\psi_{2}$ for the variational optimization
problem (\ref{eq:sigmai2}, \ref{eq:orth}) if $c_{2}\neq0$, and
thus $\left\langle \rho_{0},\psi_{2}\right\rangle _{\rho_{0}^{-1}}=0$.
By mathematical induction, $\left\langle \rho_{0},\psi_{i}\right\rangle _{\rho_{0}^{-1}}=0$
for all $\sigma_{i}>0$.
\end{proof}
Based on this proposition, $\psi_{i}$ can be approximated by
\begin{equation}
\psi_{i}=\mathbf{u}_{i}^{\top}\boldsymbol{\chi}\label{eq:psi-u}
\end{equation}
with (\ref{eq:constraint-chi}) being satisfied. Substituting the
Ansatz (\ref{eq:psi-u}) into (\ref{eq:sigmai2}, \ref{eq:orth})
and replacing expected values with empirical estimates yields
\begin{eqnarray*}
\mathbf{u}_{i} & = \arg\max_{\mathbf{u}} \frac{1}{N^{2}}\mathrm{tr}\left(\mathbf{u}^{\top}\boldsymbol{\chi}(\mathbf{X})^{\top}\mathbf{G}_{yy}\boldsymbol{\chi}(\mathbf{X})\mathbf{u}\right)\\
\mathrm{s.t.} & \mathbf{u}^{\top}\mathbf{u}=1,\quad\mathbf{u}^{\top}\mathbf{u}_{j}=0\text{ for }j=1,\ldots,i-1.
\end{eqnarray*}
This problem for all $i$ can be equivalently solved by the KVAD algorithm
in Section \ref{subsec:Approximation-with-unknown-f}. Consequently,
$s_{i},s_{i}^{-1}q_{i+1},f_{i+1}\rho_{0}$ are variational estimates
of the $i$th singular value, left singular function and right singular
function of the operator $\tilde{\mathcal{P}}_{\tau}$.

\section{Proof of (\ref{eq:diffusion-distance})\label{sec:Proof-of-distance}}

From (\ref{eq:svd}) and the orthonormality of $\phi_{i}$, we have
\begin{eqnarray*}
D_{\tau}\left(\mathbf{x},\mathbf{x}'\right)^{2} & = &\left\Vert \mathcal{P}_{\tau}\delta_{\mathbf{x}}-\mathcal{P}_{\tau}\delta_{\mathbf{x}'}\right\Vert _{\mathcal{E}}^{2}\\
 & =& \left\Vert \mathcal{\tilde{P}}_{\tau}\delta_{\mathbf{x}}-\mathcal{\tilde{P}}_{\tau}\delta_{\mathbf{x}'}\right\Vert _{\mathcal{E}}^{2}\\
 & =& \left\Vert \sum_{i}\sigma_{i}\left(\psi_{i}(\mathbf{x})\rho_{0}(\mathbf{x})^{-1}-\psi_{i}(\mathbf{x}')\rho_{0}(\mathbf{x}')^{-1}\right)\phi_{i}\right\Vert _{\mathcal{E}}^{2}\\
 & =& \sum_{i,j}\sigma_{i}\sigma_{j}\left(\psi_{i}(\mathbf{x})\rho_{0}(\mathbf{x})^{-1}-\psi_{i}(\mathbf{x}')\rho_{0}(\mathbf{x}')^{-1}\right)\\
  &\qquad\cdot & \left(\psi_{j}(\mathbf{x})\rho_{0}(\mathbf{x})^{-1}-\psi_{j}(\mathbf{x}')\rho_{0}(\mathbf{x}')^{-1}\right)\left\langle \phi_{i},\phi_{j}\right\rangle _{\mathcal{E}}\\
 & =& \sum_{i=1}^{\infty}\sigma_{i}^{2}\left(\psi_{i}(\mathbf{x})\rho_{0}(\mathbf{x})^{-1}-\psi_{i}(\mathbf{x}')\rho_{0}(\mathbf{x}')^{-1}\right)^{2}.
\end{eqnarray*}
If the KVAD algorithm gives the exact approximation of singular components
and $\sigma_{i}=0$ for $i>m-1$, we can get
\[
D_{\tau}\left(\mathbf{x},\mathbf{x}'\right)^{2}=\sum_{i=1}^{m-1}s_{i}^{2}\left(f_{i+1}(\mathbf{x})-f_{i+1}(\mathbf{x}')\right)^{2}.
\]

\section{Comparison between KVAD and conditional mean embedding\label{sec:Comparison-between-KVAD-KME}}

We consider
\[
f_{i}(\mathbf{x})=\mathbf{u}_{i}^{\top}\varphi(\mathbf{x}),
\]
for $i=1,\ldots,N$ in KVAD, and assume that $\mathbf{G}_{xx}=\left[\kappa(\mathbf{x}_{i},\mathbf{x}_{j})\right]\in\mathbb{R}^{N\times N}$
is invertible. Here $\mathbf{u}_{i}$ is the $i$th column of $\mathbf{U}$,
and $\varphi(\mathbf{x})$ is the kernel mapping and can be explicitly
represented as a function from $\mathbb{M}$ to a (possibly infinite-dimensional)
Euclidean space with $\kappa(\mathbf{x},\mathbf{y})=\varphi(\mathbf{x})^{\top}\varphi(\mathbf{y})$
\cite{minh2006mercer}. For a given data set $\{(\mathbf{x}_{n},\mathbf{y}_{n})\}_{n=1}^{N}$,
an arbitrary $\mathbf{u}_{i}$ can be decomposed as
\[
\mathbf{u}_{i}=\varphi(\mathbf{X})^{\top}\mathbf{a}_{i}+\mathbf{u}_{i}^{\perp}
\]
with $\varphi(\mathbf{X})=(\varphi(\mathbf{x}_{1}),\ldots,\varphi(\mathbf{x}_{N}))^{\top}$
and $\varphi(\mathbf{X})^{\top}\mathbf{u}_{i}^{\perp}=\mathbf{0}$,
and
\[
\mathbf{u}_{i}^{\top}\varphi(\mathbf{x}_{n})=\left(\varphi(\mathbf{X})^{\top}\mathbf{a}_{i}\right)^{\top}\varphi(\mathbf{x}_{n}),\quad\forall n.
\]
So, we can assume without loss of generality that each $\mathbf{u}_{i}$
can be represented as a linear combination of $\{\varphi(\mathbf{x}_{1}),\ldots,\varphi(\mathbf{x}_{N})\}$,
and therefore all invertible $\mathbf{A}=(\mathbf{a}_{1},\ldots,\mathbf{a}_{N})$
can generate the equivalent model. For convenience of analysis, we
set $\mathbf{A}=\mathbf{I}$ and
\[
\mathbf{f}(\mathbf{x})=\varphi(\mathbf{X})\varphi(\mathbf{x}).
\]

\[
N\mathbf{C}_{ff}=\sum_{n=1}^{N}\varphi(\mathbf{X})\varphi(\mathbf{x}_{n})\varphi(\mathbf{x}_{n})^{\top}\varphi(\mathbf{X})^{\top}=\mathbf{G}_{xx}^{2}
\]
Then

\[
\mathbf{q}(\mathbf{y})=\sum_{n}\mathbf{G}_{xx}^{-2}\varphi(\mathbf{X})\varphi(\mathbf{x}_{n})\delta_{\mathbf{y}_{n}}(\mathbf{y}),
\]
\begin{eqnarray*}
\hat{p}_{\tau}(\mathbf{x},\mathbf{y}) & =& \mathbf{f}(\mathbf{x})^{\top}\mathbf{q}(\mathbf{y})\\
 & =& \varphi(\mathbf{x})^{\top}\varphi(\mathbf{X})^{\top}\sum_{n}\mathbf{G}_{xx}^{-2}\varphi(\mathbf{X})\varphi(\mathbf{x}_{n})\delta_{\mathbf{y}_{n}}(\mathbf{y})
\end{eqnarray*}
and we can obtain
\begin{eqnarray*}
\mathbb{E}[\varphi(\mathbf{y})|\mathbf{x}] & = &\sum_{n}\varphi(\mathbf{y}_{n})\varphi(\mathbf{x}_{n})^{\top}\varphi(\mathbf{X})^{\top}\mathbf{G}_{xx}^{-2}\varphi(\mathbf{X})\varphi(\mathbf{x})\\
 & = & \varphi(\mathbf{Y})^{\top}\varphi(\mathbf{X})\varphi(\mathbf{X})^{\top}\mathbf{G}_{xx}^{-2}\varphi(\mathbf{X})\varphi(\mathbf{x})\\
 & = & \varphi(\mathbf{Y})^{\top}\mathbf{G}_{xx}^{-1}\left(\kappa(\mathbf{x}_{1},\mathbf{x}),\ldots,\kappa(\mathbf{x}_{N},\mathbf{x})\right)^{\top},
\end{eqnarray*}
which is equivalent to the result of conditional mean embedding \cite{song2013kernel}.

\section{Estimated singular components of Examples \ref{exa:oscillator} and
\ref{exa:Lorenz} with $\xi\protect\neq0$\label{sec:Estimated-singular-components}}

\begin{figure}
\begin{centering}
\includegraphics[width=0.8\textwidth]{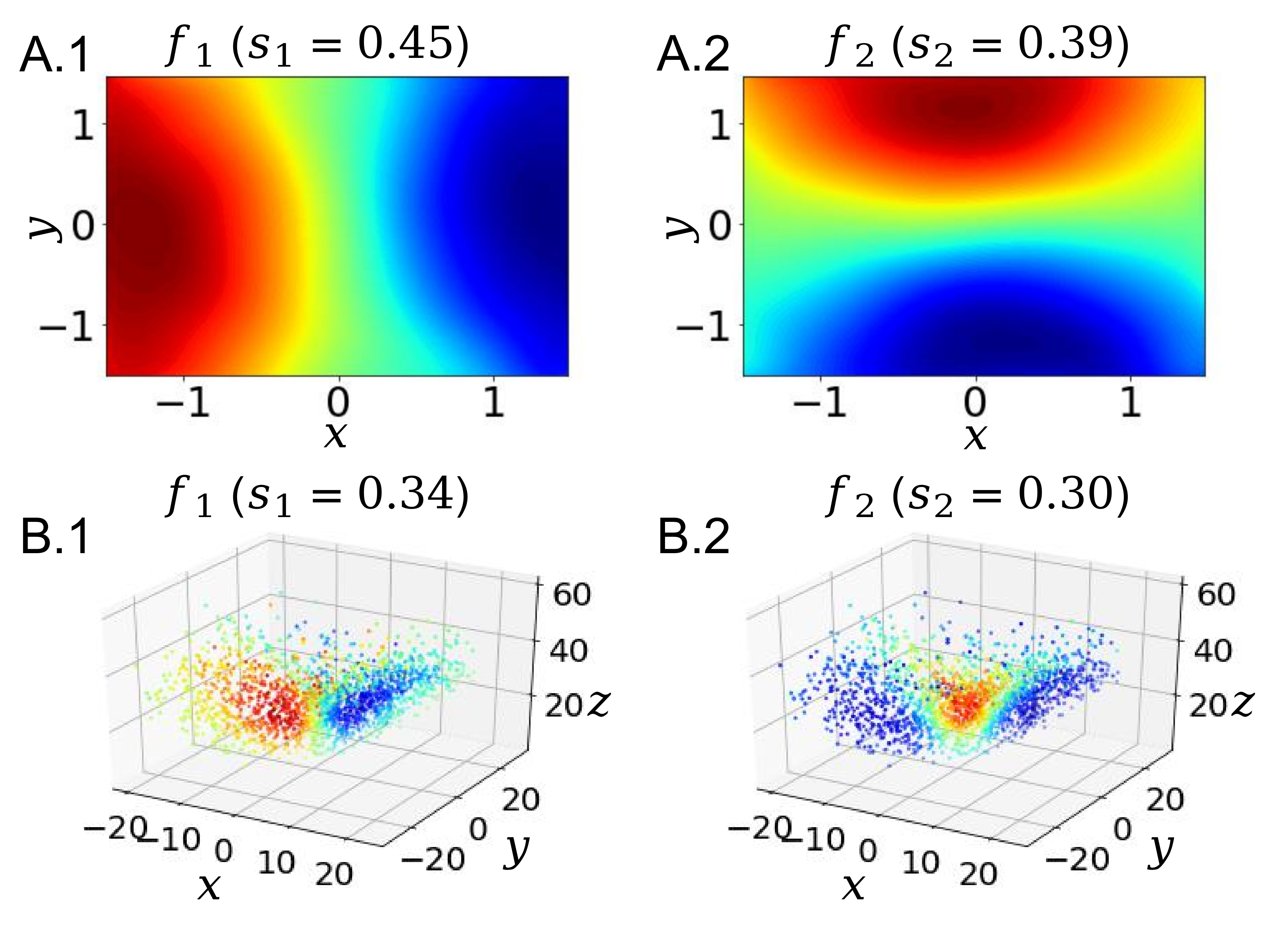}
\par\end{centering}
\caption{\label{fig:noise-singular}The first two singular components computed
by KVAD. (A) The oscillator with $\xi=0.2$. (B) The Lorenz system
with $\xi=0.5$.}
\end{figure}

\FloatBarrier

\bibliographystyle{unsrt}
\bibliography{kvad-submission.bbl}

\end{document}